\documentclass[10pt]{amsart}
\usepackage{amsfonts}
\usepackage{amsmath}
\usepackage{amssymb}
\usepackage{geometry}
\newtheorem{definition}{Definition}

\newtheorem{theorem}{Theorem}
\newtheorem{proposition}{Proposition}
\newtheorem{lemma}{Lemma}
\newtheorem{example}{Example}

\usepackage{tikz}
\usetikzlibrary{arrows}
\usepackage{boxedminipage}
\usepackage{color}
\usepackage{float,bigstrut,multirow}

\usepackage[all,line]{xy}

\DeclareMathOperator*{\ext}{ext}

\usepackage{verbatim}

\begin{document}
\allowdisplaybreaks
\title{Discrete Hamiltonian Variational Integrators}
\author{Melvin Leok}
\address{Department of Mathematics, University of California, San Diego, 9500 Gilman Drive, La Jolla, California 92093--0112.}
\email{mleok@math.ucsd.edu}
\author{Jingjing Zhang}
\address{LSEC, Academy of Mathematics and Systems Science, Chinese Academy of Sciences, Beijing 100190, China.}
\email{zhangjj@lsec.cc.ac.cn}
\date{\today}
\begin{abstract}
 We consider the continuous and discrete-time Hamilton's variational principle on phase space, and characterize the exact discrete Hamiltonian which provides an exact correspondence between discrete and continuous Hamiltonian mechanics. The variational characterization of the exact discrete Hamiltonian naturally leads to a class of generalized Galerkin Hamiltonian variational integrators, which include the symplectic partitioned Runge--Kutta methods. We also characterize the group invariance properties of discrete Hamiltonians which lead to a discrete Noether's theorem.
\end{abstract}

\maketitle

\section{Introduction}
\subsection{Discrete Mechanics}
Discrete-time analogues of Lagrangian and Hamiltonian mechanics, which are derived from discrete variational principles, yield a class of geometric numerical integrators~\cite{HaLuWa2006} referred to as variational integrators~\cite{LaWe2006, MaWe2001}. The discrete variational approach to constructing numerical integrators is of interest as they automatically yield methods that are symplectic, and by a backward error analysis, exhibit bounded energy errors for exponentially long times (see, for example, \cite{Ha1994}). When the discrete Lagrangian or Hamiltonian is group-invariant, they will yield numerical methods that are momentum preserving.

Discrete Hamiltonian mechanics can be derived from discrete Lagrangian mechanics by relaxing the discrete second-order curve condition. The dual formulation of this constrained optimization problem yields discrete Hamiltonian mechanics \cite{LaWe2006}. Alternatively, the second-order curve condition can be imposed using Lagrange multipliers, and this corresponds to the discrete Hamilton--Pontryagin principle~\cite{LeOh2008}. 

In contrast to the prior literature on discrete Hamiltonian mechanics, which typically start from the Lagrangian setting, we will focus on constructing Hamiltonian variational integrators from the Hamiltonian point of view, without recourse to the Lagrangian formulation. When the Hamiltonian is hyperregular, it is possible to obtain the corresponding Lagrangian function, adopt the Galerkin construction of Lagrangian variational integrators to obtain a discrete Lagrangian, and then perform a discrete Legendre transformation to obtain a discrete Hamiltonian. This is described in the following diagram:
\[\xymatrix@!0@R=0.75in@C=1.35in{H(q, p)\ar[r]^{\mathbb{F}H}\ar@{-->}[d] & L(q,\dot{q})\ar[d]\\
H_d^+(q_0,p_1)&L_d(q_0,q_1)\ar[l]^{\mathbb{F}L_d}}\]
The goal of this paper is to directly express the discrete Hamiltonian in terms of the continuous Hamiltonian, so that the diagram above commutes when the Hamiltonian is hyperregular. An added benefit is that such an approach would remain valid even if the Hamiltonian is degenerate, as is the case for point vortices (see \cite{Ne2001}, p.\,22), and no corresponding Lagrangian formulation exists.

The Galerkin construction for Lagrangian variational integrators is attractive, since it provides a general framework for constructing a large class of symplectic methods based on suitable choices of finite-dimensional approximation spaces, and numerical quadrature formulas. Our approach allows one to apply the Galerkin construction of variational integrators to Hamiltonian systems directly, and may potentially generalize to variational integrators for multisymplectic Hamiltonian PDEs \cite{BrRe2001, LeMaOrWe2003, MaPeSh1999}.

Discrete Lagrangian mechanics is expressed in terms of a discrete Lagrangian, which can be viewed as a Type I generating function of a symplectic map, and discrete Hamiltonian mechanics is naturally expressed in terms of discrete Hamiltonians {\cite{LaWe2006}, which are either Type II or III generating functions. The discrete Hamiltonian perspective allows one to avoid some of the technical difficulties associated with the singularity associated with Type I generating functions at time $t=0$ (see \cite{Ma1992}, p.\,177). 

\begin{example}
To illustrate the difficulties associated with degenerate Hamiltonians, consider
\[H(q,p)=qp,\]
with Legendre transformation given by $\mathbb{F}H:T^*Q\rightarrow TQ$, $(q,p)\mapsto(q, \partial H/\partial p)=(q, q)$. Clearly, in this situation, the Legendre transformation is not invertible. Furthermore, the associated Lagrangian is identically zero, i.e., $L(q,\dot q)=\left.p\dot q - H(q,p)\right|_{\dot q = \partial H/\partial p}=\left.p\dot{q}-qp\right|_{\dot{q}=q}\equiv 0$.

The associated Hamilton's equations is given by $\dot{q}=\partial H/\partial p=q$, $\dot{p}=-\partial H/\partial q=-p$, with exact solution $q(t)=q(0) \exp (t)$, $p(t)=p(0)\exp(-t)$. This exact solution is, in general, incompatible with the $(q_0, q_1)$ boundary conditions associated with Type I generating functions, but it is compatible with the $(q_0,p_1)$ boundary conditions associated with Type II generating functions.

In view of this example, our discussion of discrete Hamiltonian mechanics will be expressed directly in terms of continuous Hamiltonians and Type II generating functions.
\end{example}

\subsection{Main Results} We provide a characterization of the Type II generating function that generates the exact flow of Hamilton's equations, and derive the corresponding Type II Hamilton--Jacobi equation that it satisfies. By considering a discrete Type II Hamilton's variational principle in phase space, we derive the discrete Hamilton's equations in terms of a discrete Hamiltonian. We provide a variational characterization of the exact discrete Hamiltonian that, when substituted into the discrete Hamilton's equations, generates samples of the exact continuous solution of Hamilton's equations. Also, we introduce a discrete Type II Hamilton--Jacobi equation.

From the variational characterization of the exact discrete Hamiltonian, we introduce a generalized Galerkin approximation from both the Hamiltonian and Lagrangian sides, and show that they are equivalent when the Hamiltonian is hyperregular. In addition, we provide a systematic means of implementing these methods as symplectic-partitioned Runge--Kutta (SPRK) methods. We also establish the invariance properties of the discrete Hamiltonian that yield a discrete Noether's theorem. Galerkin discrete Hamiltonians derived from group-invariant interpolatory functions satisfy these invariance properties, and therefore preserve momentum.

\subsection{Outline of the Paper} In Section \ref{sec:continuous_mechanics}, we present the Type II analogues of Hamilton's phase space variational principle and the Hamilton--Jacobi equation, and we consider the discrete-time analogues of these in Section \ref{sec:discrete_mechanics}. In Section \ref{sec:galerkin_vi}, we develop generalized Galerkin Hamiltonian and Lagrangian variational integrators, and consider their implementation as symplectic-partitioned Runge--Kutta methods. In Section \ref{sec:momentum_preservation}, we establish a discrete Noether's theorem, and provide a discrete Hamiltonian that preserves momentum.

\section{Variational Formulation of Hamiltonian Mechanics}\label{sec:continuous_mechanics}
\subsection{Hamilton's Variational Principle for Hamiltonian and Lagrangian Mechanics} Considering a $n$-dimensional configuration manifold $Q$ with associated tangent space $TQ$ and phase space $T^*Q.$ We introduce generalized coordinates $q=(q^1, q^2, \ldots, q^n)$ on $Q$ and $(q, p)=(q^1, q^2, \ldots, q^n, p_1, p_2, \ldots, p_n)$ on $T^*Q.$ Given a Hamiltonian $H:T^*Q\rightarrow \mathbb{R}$, Hamilton's phase space variational principle states that
\[\delta \int_0^T [p\dot{q} -H(q,p) ]dt=0,\]
for fixed $q(0)$ and $q(T)$. This is equivalent to Hamilton's canonical equations,
\begin{align}\label{Hameq}
\dot{q}=\frac{\partial H}{\partial p}(q, p), \quad
\dot{p}=-\frac{\partial H}{\partial q}(q, p).
\end{align}
If the Hamiltonian is hyperregular, there is a corresponding Lagrangian $L:TQ\rightarrow \mathbb{R}$, given by
\[ L(q,\dot q)=\ext_p \, p\dot q-H(q,p)=\left.p\dot q - H(q,p)\right|_{\dot q = \partial H/\partial p},\]
where $\ext_p$ denotes the extremum over $p$. Then, 
% and
Hamilton's phase space principle is equivalent to Hamilton's principle,
\[\delta \int_0^T L(q,\dot q) dt =0,\]
for fixed $q(0)$ and $q(T)$. The exact discrete Lagrangian is then given by,
\[ L_d^\text{exact}(q_0, q_1) = \ext_{\substack{q\in C^2([0,h],Q) \\ q(0)=q_0, q(h)=q_1}} \int_0^h L(q(t), \dot q(t)) dt =  \ext_{\substack{(q,p)\in C^2([0,h],T^*Q) \\ q(0)=q_0, q(h)=q_1}} \int_0^h p\dot{q} -H(q,p) dt,\]
which correspond to Jacobi's solution of the Hamilton--Jacobi equation. The usual characterization of the exact discrete Lagrangian involves evaluating the action integral on a curve $q$ that satisfies the boundary conditions at the endpoints, and the Euler--Lagrange equations in the interior, however, as we will see, the variational characterization above naturally leads to the construction of Galerkin variational integrators.

\subsection{Type II Hamilton's Variational Principle in Phase Space} The boundary conditions associated with both Hamilton's and Hamilton's phase space variational principle are naturally related to Type I generating functions, since they specify the positions at the initial and final times. We will introduce a version of Hamilton's phase space principle for fixed $q(0)$, $p(T)$ boundary conditions, that correspond to a Type II generating function, which we refer to as the Type II Hamilton's variational principle in phase space. As would be expected, this will give a characterization of the exact discrete Hamiltonian. Taking the Legendre transformation of the Jacobi solution of the Hamilton--Jacobi equation leads us to consider the following functional, $\mathfrak{S}:C^2([0,T],T^*Q)\rightarrow \mathbb{R}$,
\begin{align}\label{Htype2bv}
\mathfrak{S}(q(\cdot),p(\cdot))=p(T) q(T) - \int_0^T [p \dot{q}-H(q(t),
p(t))] dt.
\end{align}

\begin{lemma}\label{lemma1}
Consider the action functional $\mathfrak{S}(q (\cdot), p(\cdot))$ given by \eqref{Htype2bv}. The condition that $\mathfrak{S}(q (\cdot), p(\cdot))$ is stationary with respect to boundary conditions $\delta q(0)=0$, $\delta p(T)=0$  is equivalent to $(q (\cdot), p(\cdot))$ satisfying
Hamilton's canonical equations \eqref{Hameq}.
\end{lemma}

\begin{proof}
Direct computation of the variation of $\mathfrak{S}$ over the path space $C^2([0,T],T^*Q)$ yields,
\begin{align}
\nonumber \delta \mathfrak{S}&=q(T)\delta p(T) + p(T) \delta q(T)\\
\nonumber&\qquad-\int_0^T \left[\dot{q}(t)\delta p(t) + p(t)\delta \dot{q}(t) - \frac{\partial H}{\partial q}\left(q(t), p(t)\right)\delta q(t) -\frac{\partial H}{\partial p}\left(q(t),
p(t)\right)\delta p(t) \right] dt.
\intertext{By using integration by  parts and  the boundary conditions $\delta
q(0)=0, \delta p(T)=0,$ we obtain,}
\label{H2extre}
\delta \mathfrak{S}&=q(T)\delta p(T) + p(T) \delta q(T)-p(T) \delta q(T)+ p(0) \delta q(0)\\
\nonumber&\qquad+\int_0^T \left[
\left(\dot{p}(t)+\frac{\partial H}{\partial q}(q(t),
 p(t))\right)\delta q(t) -\left(\dot{q}(t)-\frac{\partial H}{\partial p}(q(t),
 p(t))\right)\delta p(t)\right]dt\\
 \nonumber &=\int_0^T  \left[ \left(\dot{p}(t)+\frac{\partial H}{\partial q}(q(t), p(t))\right)\delta q(t)-\left(\dot{q}(t)-\frac{\partial H}{\partial p}(q(t), p(t))\right)\delta p(t)\right] dt.
\intertext{If $(q, p) $ satisfies Hamilton's equations
\eqref{Hameq}, the integrand vanishes, and $\delta \mathfrak{S}=0.$
Conversely, if we assume that $ \delta \mathfrak{S}=0$ for any
$\delta q(0)=0, \delta p(T)=0,$ then from \eqref{H2extre}, we
obtain} \nonumber\delta \mathfrak{S}&=\int_0^T
\left[\left(\dot{p}(t)+\frac{\partial H}{\partial q}(q(t),
 p(t))\right)\delta q(t)-\left(\dot{q}(t)-\frac{\partial H}{\partial p}(q(t),
 p(t))\right)\delta p(t)\right]dt=0,
 \end{align}
and by the fundamental theorem of calculus of variations \cite{Ar1989}, we
recover Hamilton's equations,
\[\dot{q}(t)=\frac{\partial H}{\partial p}\left(q(t),
 p(t)\right),\qquad
 \dot{p}(t)=-\frac{\partial H}{\partial q}\left(q(t),
 p(t)\right).\]
\end{proof}
The above lemma states that the integral curve $(q(\cdot),p(\cdot))$ of Hamilton's equations extremizes the action functional $\mathfrak{S}(q(\cdot),p(\cdot))$ \eqref{Htype2bv} for fixed boundary conditions
$q(0)$, $p(T)$. We now introduce the function $\mathcal{S}(q_0, p_T)$, which is given by the extremal value of the action functional $\mathfrak{S}$ over the family of curves satisfying the boundary conditions $q(0)=q_0$, $p(T)=p_T$,
\begin{align}\label{2map}
\mathcal{S}(q_0, p_T)&= \ext_{\substack{(q, p) \in
C^2([0, T],T^*Q)\\q(0)=q_0, p(T)=p_T}} \mathfrak{S}(q(\cdot),p(\cdot)) =  \ext_{\substack{(q, p) \in
C^2([0, T],T^*Q)\\q(0)=q_0, p(T)=p_T}} p_T q_T - \int_0^T \left[ p \dot{q}-H(q, p) \right]
dt.
\end{align}
The next theorem describes how $\mathcal{S}(q_0, p_T)$ generates the flow of Hamilton's equations.
\begin{theorem}\label{thm1}
Given the  function $\mathcal{S}(q_0, p_T)$ defined by \eqref{2map},
the exact time-$T$ flow map of Hamilton's equations
$(q_0, p_0) \mapsto (q_T, p_T)$ is implicitly given by the following relation,
\begin{align}\label{defp0q1}
q_T=D_{2}\mathcal{S}(q_0, p_T), \qquad p_0=D_{1} \mathcal{S}(q_0, p_T).
\end{align}
In particular, $\mathcal{S}(q_0, p_T)$ is a Type II generating function that generates the exact flow of Hamilton's equations.
\end{theorem}
\begin{proof}We directly compute
\begin{align*}
\frac{\partial \mathcal{S}}{\partial q_0}(q_0, p_T)&=\frac{\partial  q_T}{\partial q_0} p_T-\int_0^T
\left[\frac{\partial p(t)}{\partial q_0} \dot{q}(t)+\frac{\partial \dot{q}(t)}{\partial q_0}p(t)-\frac{\partial q(t)}{\partial q_0}\frac{\partial H}{\partial q}(q, p)-\frac{\partial p(t)}{\partial q_0}\frac{\partial H}{\partial p}(q, p) \right]dt \\
&=\frac{\partial q_T}{\partial q_0} p_T-\frac{\partial
q_T}{\partial q_0} p_T+\frac{\partial q_0}{\partial q_0}
p_0-\int_0^T \left[\frac{\partial p(t)}{\partial
q_0}\left(\dot{q}-\frac{\partial H}{\partial p}(q, p)\right)
-\frac{\partial q(t)}{\partial q_0}\left(\dot{p}+\frac{\partial H}{\partial q}(q, p)\right)\right]dt \\
&=p_0-\int_0^T \left[\frac{\partial p(t)}{\partial
q_0}\left(\dot{q}-\frac{\partial H}{\partial p}(q, p)\right) -\frac{\partial q(t)}{\partial
q_0}\left(\dot{p}+\frac{\partial H}{\partial q}(q, p)\right)\right] dt,
\intertext{where we used integration by parts. By Lemma \ref{lemma1}, the extremum of $\mathfrak{S}$ is achieved when the curve $(q,p)$ satisfies Hamilton's equations. Consequently, the integrand in the above equation vanishes, giving $p_0=\frac{\partial \mathcal{S}}{\partial q_0}(q_0, p_T).$
Similarly, by using integration by parts, and restricting ourselves to curves $(q,p)$ which satisfy Hamilton's equations, we obtain}
\frac{\partial \mathcal{S}}{\partial p_T}(q_0, p_T) &=\frac{\partial  p_T}{\partial p_T} q_T+  \frac{\partial q_T}{\partial p_T}p_T-\int_0^T \left[\frac{\partial p(t)}{\partial
p_T} \dot{q}(t) +\frac{\partial \dot{q}(t)}{\partial p_T}p(t)-\frac{\partial q(t)}{\partial p_T}\frac{\partial H}{\partial q}(q, p)
-\frac{\partial p(t)}{\partial p_T}\frac{\partial H}{\partial p}(q, p) \right]dt \\
&=q_T+\frac{\partial q_T}{\partial p_T}p_T -\frac{\partial
q_T}{\partial p_T}p_T +\frac{\partial q_0}{\partial p_T}p_0
-\int_0^T \left[ \frac{\partial p(t)}{\partial p_T}
\left(\dot{q}-\frac{\partial H}{\partial p}(q, p)\right)-\frac{\partial
q(t)}{\partial p_T}\left(\dot{p}+\frac{\partial H}{\partial q}(q, p)\right) \right]dt \\
&=q_T.
\end{align*}
\end{proof}

\subsection{Type II Hamilton--Jacobi Equation} Let us explicitly consider $\mathcal{S}(q_0, p_T)$ as a function of the time $T$, which we denote by $\mathcal{S}_T(q_0, p_T)$. Theorem \ref{thm1} states that the Type II generating function $\mathcal{S}_T(q_0, p_T)$ generates the exact time-$T$ flow map of Hamilton's equations, and consequently it has to be related by the Legendre transformation to the Jacobi solution of the Hamilton--Jacobi equation, which is the Type I generating function for the same flow map.
Consequently, we expect that the function $\mathcal{S}_T(q_0, p_T)$ satisfies a Type II analogue of the Hamilton--Jacobi equation, which we derive in the following proposition.
\begin{proposition}
Let 
\begin{align}\label{type2sol}
S_2(q_0, p, t)\equiv \mathcal{S}_t(q_0, p)=\ext_{\substack{(q, p) \in
C^2([0, t],T^*Q)\\q(0)=q_0, p(t)=p}}\left (p(t)  q(t) - \int_0^t \left[ p(s) \dot{q}(s)-H(q(s), p(s)) \right] ds\right).
\end{align}
Then, the function $S_2(q_0, p, t)$ satisfies the \textbf{Type II Hamilton--Jacobi equation},
\begin{align} \label{type2hj}
\frac{\partial S_2( q_0, p, t)}{\partial t} =H\left(\frac{\partial S_2}{\partial p},p\right).
\end{align}
\end{proposition}
\begin{proof}
From the definition of $S_2(q_0, p, t),$ the curve that extremizes the functional connects the fixed initial point $(q_0, p_0)$ with the arbitrary final point $(q,p)$ at time $t$. Computing the time derivative of $S_2(q_0, p, t)$ yields,
\begin{align}\label{Sleft}
\frac{d S_2}{d t}=\dot{p}(t) q(t) + p(t) \dot{q}(t)-p(t)\dot{q}(t)+H(q(t), p(t)).
\end{align}
On the other hand,
\begin{align}\label{Sright}
\frac{d S_2}{d t}= \dot{p}(t)\frac{\partial S_2}{\partial p}+\frac{\partial S_2}{\partial t}  .
\end{align}
Equating \eqref{Sleft} and \eqref{Sright}, and applying \eqref{defp0q1} yields
\begin{align}
\frac{\partial S_2}{\partial t}&=\dot{p}(t) q(t) +H(q(t), p(t))-\dot{p}(t)\frac{\partial S_2}{\partial p}=H(q(t), p(t))=H\left( \frac{\partial S_2}{\partial p}, p\right).
\end{align}
\end{proof}
%\begin{remark}
The Type II Hamilton--Jacobi equation also appears on p.\,201 of \cite{HaLuWa2006} and in \cite{LeDiMe2007}. However, this equation has generally been used in the construction of symplectic integrators based on Type II generating functions by considering a series expansion of $S_2$ in powers of $t$, substituting the series in the Type II Hamilton--Jacobi equation, and truncating. Then, a term-by-term comparison allows one to determine the coefficients in the series expansion of $S_2$, from which one constructs a symplectic map that approximates the exact flow map \cite{ChSc1990, Fe1986, Ru1983, Vo1956}.

However, approximating Jacobi's solution on the Lagrangian side, or the exact discrete right Hamiltonian $\mathcal{S}(q_0, p_T)$ in \eqref{2map}, in terms of their variational characterization provides an elegant method for constructing symplectic integrators. In particular, this naturally leads to the generalized Galerkin framework for constructing discrete Lagrangians and discrete Hamiltonians, which we will explore in the rest of the paper.

In Section \ref{sec:discrete_mechanics}, we will also present a discrete analogue of the Type II Hamilton--Jacobi equation, which can be viewed as a composition theorem that expresses the discrete Hamiltonian for a given time interval in terms of discrete Hamiltonians for the subintervals. This can be viewed as the Type II analogue of the discrete Hamilton--Jacobi equation that was introduced in \cite{ElSc1996}.
%\end{remark}

\section{Discrete Variational Hamiltonian Mechanics}\label{sec:discrete_mechanics}
\subsection{Discrete Type II Hamilton's Variational Principle in Phase Space}
The Lagrangian formulation of discrete variational mechanics is based on a discretization of Hamilton's principle, and a comprehensive review of this approach is given in \cite{MaWe2001}. The Hamiltonian analogue of discrete variational mechanics was introduced in \cite{LaWe2006}, wherein discrete Lagrangian mechanics was viewed as the primal formulation of a constrained discrete optimization problem, where the constraints are given by the discrete analogue of the second-order curve condition, and dual formulation of this yields discrete Hamiltonian variational mechanics. An analogous approach is based on the discrete Hamilton--Pontryagin variational principle~\cite{LeOh2008}, in which the discrete Hamilton's principle is augmented with a Lagrange multiplier term that enforces the discrete second-order curve condition.

We begin by introducing a partition of the time interval $[0,T]$ with the discrete times $0=t_0<t_1<\ldots<t_N=T$, and a discrete curve in $T^*Q$, denoted by $\{(q_k, p_k)\}_{k=0}^N$, where $q_k\approx q(t_k)$, and $p_k\approx p(t_k)$. Our discrete variational principle will be formulated in terms of a discrete Hamiltonian $H_d^+(q_k, p_{k+1})$, which is an approximation of the Type II generating function given in \eqref{2map},
\begin{align}\label{dishd}H_d^+(q_k, p_{k+1})\approx  \ext_{\substack{(q, p) \in
C^2([t_k,t_{k+1}],T^*Q)\\q(t_k)=q_k, p(t_{k+1})=p_{k+1}}} p(t_{k+1}) q (t_{k+1}) - \int_{t_k}^{t_{k+1}} \left[ p \dot{q}-H(q, p) \right]
dt.
\end{align}

As we saw in Section \ref{sec:continuous_mechanics}, the curve in phase space with fixed boundary conditions $(q_0,p_T)$ that extremizes the functional \eqref{Htype2bv},
\[
\mathfrak{S}(q(\cdot),p(\cdot))=p(T) q(T) - \int_0^T [p \dot{q}-H(q(t),
p(t))] dt,
\]
satisfies Hamilton's canonical equations. Consequently, we can formulate discrete variational Hamiltonian mechanics in terms of a discrete analogue of this functional, which is given by,
\begin{align}\label{disHtype2bv}
\mathfrak{S}_d(\{(q_k, p_k)\}_{k=0}^N)&=p_N q_N -\sum_{k=0}^{N-1}
\int_{t_k}^{t_{k+1}}\left[ p \dot{q}-H(q(t), p(t))\right] dt \\
\nonumber
 &=p_N q_N -\sum_{k=0}^{N-1}\left[p_{k+1} q_{k+1} -H_d^+(q_k,
 p_{k+1})\right].
\end{align}
Then, the \textit{Type II discrete Hamilton's phase space variational principle} states that $\delta \mathfrak{S}_d(\{(q_k, p_k)\}_{k=0}^N)=0$ for discrete curves in $T^*Q$ with fixed $(q_0, p_N)$ boundary conditions.
\begin{lemma}\label{lemma:discrete_hameq}
The Type II discrete Hamilton's phase space variational principle is equivalent to the discrete right
Hamilton's equations
\begin{equation}\label{dish2}
\begin{alignedat}{2} 
  q_{k}&=D_{2} H_d^+(q_{k-1}, p_{k}),&\qquad k&=1, \ldots, 
N-1, \\
p_k&= D_{1} H_d^+(q_k, p_{k+1}),&\qquad k&=1, \ldots, N-1,
\end{alignedat}
\end{equation}
where $H_d^+(q_k, p_{k+1})$ is defined in \eqref{dishd}.
\end{lemma}
\begin{proof}We compute the variation of $\mathfrak{S}_d$,
\begin{align*}
\delta \mathfrak{S}_d& =\delta\left( p_N q_N
-\sum_{k=0}^{N-1}(p_{k+1} q_{k+1} -H_d^+(q_k,
 p_{k+1}))\right) \\ \nonumber &=\delta
\left(-\sum_{k=0}^{N-2}p_{k+1} q_{k+1}+\sum_{k=0}^{N-1} H_{d}^+(q_k,
p_{k+1})\right) \\ \nonumber &= -\sum_{k=0}^{N-2}(q_{k+1} \delta
p_{k+1} +p_{k+1}\delta q_{k+1} )+\sum_{k=0}^{N-1} \left(D_{1}
H_{d}^+(q_k,
p_{k+1}) \delta q_k+ D_{2} H_{d}^+(q_k, p_{k+1})\delta p_{k+1}\right) \\
\nonumber &= -\sum_{k=1}^{N-1}(q_{k} \delta p_k + p_k\delta q_{k}
)+\sum_{k=1}^{N-1} D_{1} H_{d}^+(q_k, p_{k+1})
\delta q_k + D_{1} H_{d}^+(q_0, p_{1})\delta q_0  \\
\nonumber &\qquad+ \sum_{k=1}^{N-1} D_{2} H_{d}^+(q_{k-1},
p_{k})\delta p_{k}+D_{2} H_{d}^+(q_{N-1}, p_{N})\delta p_N
\\ \nonumber &= -\sum_{k=1}^{N-1}\left(q_{k} -D_{2}
H_{d}^+(q_{k-1}, p_{k})\right)\delta p_{k}-\sum_{k=1}^{N-1}
\left(p_k-D_{1} H_{d}^+(q_k, p_{k+1})\right) \delta q_k \\
\nonumber & \qquad+ D_{1} H_{d}^+(q_0, p_{1})\delta q_0 +
D_{2} H_{d}^+(q_{N-1}, p_{N}) \delta p_N.
\end{align*}
where we reindexed the sum, which is the discrete analogue of integration by parts. Using the fact that $(q_0,p_N)$ are fixed, which implies $\delta q_0=0$, $\delta p_N=0$, the above
equation reduces to
\begin{align}\label{disH2extre}
\delta \mathfrak{S}_d=-\sum_{k=1}^{N-1}\left(q_{k}
-D_{2} H_{d}^+(q_{k-1}, p_{k})\right)\delta p_{k}-\sum_{k=1}^{N-1}
\left(p_k-D_{1} H_{d}^+(q_k, p_{k+1})\right) \delta q_k.
\end{align}
Clearly, if the discrete right Hamilton's equations, $q_{k}=D_{2} H_{d}^+(q_{k-1}, p_{k})$, $p_k=D_{1} H_{d}^+(q_k, p_{k+1})$, are satisfied, then the functional is stationary. Conversely, if the functional is stationary, a discrete analogue of the fundamental theorem of the calculus of variations yields the discrete right Hamilton's equations.
 \end{proof}
The above lemma states that the discrete-time solution trajectory of the discrete right Hamilton's  equations \eqref{dish2} extremizes the discrete functional \eqref{disHtype2bv} for
fixed $q_0, p_N$. However, it does not indicate how the discrete solution is related to $p_0$, $q_N$. Note that the discrete solution trajectory that renders $\mathfrak{S}_d((\{(q_k, p_k)\}_{k=0}^N)$ stationary depends on the boundary conditions $q_0$, $p_N$. Consequently, we can introduce the function $\mathcal{S}_d$ which is given by the extremal value of the discrete functional $\mathfrak{S}_d$ as a function of the boundary conditions $q(t_0), p(t_N)$, and is explicitly given by
\begin{align}\label{defdis2map}
\mathcal{S}_d(q(t_0), p(t_N)) &=  \ext_{\substack{ (q_k, p_k) \in T^*Q\\ q_0=q(t_0), p_N=p(t_N)}}
\mathfrak{S}_d(\{(q_k, p_k)\}_{k=0}^N) \\
\nonumber &= \ext_{\substack{ (q_k, p_k) \in T^*Q\\ q_0=q(t_0), p_N=p(t_N)}}  p_N q_N
-\sum_{k=0}^{N-1}(p_{k+1} q_{k+1} -H_d^+(q_k,
 p_{k+1})).
\end{align}
Then, using a similar approach to the proof of Theorem \ref{thm1}, we compute the
derivatives of $\mathcal{S}_d(q_0, p_N) $ with respect to $q_0,
p_N.$ By reindexing the sum, which is the discrete analogue of integration by parts, we  obtain
\begin{align*}
\frac{\partial \mathcal{S}_d}{\partial q_0}(q_0, p_N)  &=\frac{\partial }{\partial q_0}\left(-\sum_{k=0}^{N-2}p_{k+1}q_{k+1}+\sum_{k=0}^{N-1} H_{d}^+(q_k, p_{k+1})\right) \\
&=-\sum_{k=1}^{N-1}\frac{\partial p_{k}}{\partial q_0}\left(q_{k}
-D_{2} H_{d}^+(q_{k-1}, p_{k})\right)-\sum_{k=1}^{N-1}
\frac{\partial q_{k}}{\partial q_0} \left(p_k-D_{1} H_{d}^+(q_k,
p_{k+1})\right) +D_{1} H_{d}^+(q_0, p_{1}).
\end{align*}
By Lemma \ref{lemma:discrete_hameq}, the extremum of $\mathfrak{S}_d$
is obtained if the discrete curve satisfies the discrete right
Hamilton's equations \eqref{dish2}. Thus, by the definition of
$\mathcal{S}_d(q_0, p_N)$, the above equation reduces to
 \begin{align}\label{bv1} D_{1} \mathcal{S}_d(q_0, p_N)=D_{1}
H_{d}^+(q_0, p_{1}).
\end{align}
A similar argument yields
\begin{align}\label{bv2}
\frac{\partial \mathcal{S}_d}{\partial p_N}(q_0, p_N) &=\frac{\partial }{\partial p_N}\left(-\sum_{k=0}^{N-2}p_{k+1}q_{k+1}+\sum_{k=0}^{N-1} H_{d}^+(q_k, p_{k+1})\right) \\
\nonumber &=-\sum_{k=1}^{N-1}\frac{\partial p_{k}}{\partial
p_N}\left(q_{k} -D_{2} H_{d}^+(q_{k-1},
p_{k})\right)-\sum_{k=1}^{N-1} \frac{\partial
q_{k}}{\partial p_N}\left(p_k-D_{1}
H_{d}^+(q_k, p_{k+1})\right) +D_{2} H_{d}^+(q_{N-1}, p_{N})\\
\nonumber &=D_{2} H_{d}^+(q_{N-1}, p_{N}).
\end{align}
Recall that the exact discrete Hamiltonian $\mathcal{S}(q_0, p_T)$ defined in \eqref{2map} is a Type II generating function of the symplectic map, implicitly defined by the relation \eqref{defp0q1}, that is the exact flow map of the continuous Hamilton's equations. To be consistent with this, we require $\mathcal{S}_d(q_0, p_N)$ satisfies the relation
\eqref{defp0q1}, which is to say
\begin{alignat}{3}\label{disdefp0q1}
q_N&=D_{2} \mathcal{S}_d(q_0,
p_N),&\qquad p_0&=D_{1} \mathcal{S}_d(q_0, p_N).
\intertext{Comparing \eqref{bv1}--\eqref{bv2} and \eqref{disdefp0q1} we obtain}
\label{disbv}
q_N&=D_{2} H_{d}^+(q_{N-1}, p_{N}),&\qquad p_0&=D_{1} H_{d}^+(q_0, p_{1}).
\end{alignat}
Then, by combining \eqref{dish2} and \eqref{disbv}, we obtain
the complete set of discrete right Hamilton's  equations
\begin{subequations}\label{final}
\begin{alignat}{2}
q_{k+1}&=D_{2} H_d^+(q_k, p_{k+1}), &\qquad k&=0, 1, \ldots, N-1, \label{final1}\\
p_k&= D_{1}H_d^+(q_k, p_{k+1}), &\qquad k&=0, 1, \ldots, N-1. \label{final2}
\end{alignat}
\end{subequations}
It is easy to see that
\begin{align*}
0&=ddH_d^+(q_k, p_{k+1})=d\left(D_1 H_d^+(q_k, p_{k+1})d q_k+D_2 H_d^+(q_k, p_{k+1})dp_{k+1}\right)\\
&=d\left(p_k dq_k+q_{k+1}dp_{k+1}\right)=dp_k \wedge dq_k -dp_{k+1} \wedge dq_{k+1},
\end{align*}
for $k=0, \dots, N-1$. Then, successively applying above equation gives
$$dp_0 \wedge dq_0 =dp_1 \wedge dq_1= \cdots =dp_{N-1} \wedge dq_{N-1} =dp_{N}\wedge dq_N.$$
This implies that the map from the initial state $(q_0, p_0)$ to the final
state $(q_N, p_N)$ defined by \eqref{disdefp0q1} is sympletic, since it is the composition  of $N$ symplectic maps $(q_k, p_k)\mapsto (q_{k+1}, p_{k+1}),
k=0, \ldots, N-1,$ which are given by \eqref{final1}--\eqref{final2}. Alternatively, one can directly prove symplecticity of the map $(q_0,p_0)\mapsto(q_N, p_N)$ by using \eqref{disdefp0q1} to compute $0=d^2 \mathcal{S}_d(q_0,p_N)=dp_0\wedge dq_0-dp_N\wedge dq_N$. Given initial conditions $q_0, p_0$, and under the regularity assumption $\left|\frac{\partial^2 H_d^+}{\partial q_k
\partial p_{k+1}}(q_k, p_{k+1})\right|\neq 0$, we can solve \eqref{final2} to obtain $p_1,$ then substitute $p_1$ into
\eqref{final1} to get $q_1.$ By repeatedly applying this process, we
obtain the discrete solution trajectory $\{(q_k, p_k)\}_{k=1}^N.$

\subsection{Discrete Type II Hamilton--Jacobi Equation} A discrete analogue of the Hamilton--Jacobi equation was first introduced in \cite{ElSc1996}, and the connections to discrete Hamiltonian mechanics, and discrete optimal control theory were explored in \cite{OhBlLe2009}. In essence, the discrete Hamilton--Jacobi equation therein can be viewed as a composition theorem that relates the discrete Hamiltonians that generate the maps over subintervals, with the discrete Lagrangian that generates the map over the entire time interval.

We will adopt the derivation of the discrete Hamilton--Jacobi equation in \cite{OhBlLe2009}, which is based on introducing a discrete analogue of Jacobi's solution, to the setting of Type II generating functions.

\begin{theorem}
  \label{thm:DJS}
  Consider the discrete extremum function \eqref{defdis2map}:
  \begin{equation}
    \label{eq:S_d+}
  \mathcal{S}_d^k(p_k)
    = p_k q_k - \sum_{l=0}^{k-1}\left[p_{l+1}q_{l+1}-H_d^+(q_l,p_{l+1})\right],
  \end{equation}
which can be obtained from the discrete functional \eqref{disHtype2bv} by evaluating it along a solution of the right discrete Hamilton's equations \eqref{final}. Each $\mathcal{S}_d^k(p_k)$ is viewed as a function of the momentum $p_k$ at the discrete end time $t_k$. Then, these satisfy the discrete Type II Hamilton--Jacobi equation:
\begin{equation}
\label{eq:dHJ}
\mathcal{S}_d^{k+1}(p_{k+1})-\mathcal{S}_d^k(p_k)=H_d^+(D\mathcal{S}_d^k(p_k),p_{k+1})-p_k\cdot D\mathcal{S}_d^k(p_k).
\end{equation}
\end{theorem}

\begin{proof}
  From Eq.~\eqref{eq:S_d+}, we have
  \begin{equation}
    \label{eq:diffS_d+}
    \mathcal{S}_{d}^{k+1}(p_{k+1}) - \mathcal{S}_{d}^{k}(p_{k}) = H_d^+(q_k, p_{k+1})-p_k\cdot q_k,
  \end{equation}
  where $q_k$ is considered to be a function of $p_{k}$ and $p_{k+1}$, i.e., $q_{k} = q_{k}(p_{k}, p_{k+1})$.
  Taking the derivative of both sides with respect to $p_k$, we have
  \begin{equation*}
    -D\mathcal{S}_{d}^{k}(p_k) = -q_k + \frac{\partial q_k}{\partial p_k} \cdot \left[ D_1 H_d^+(q_k, p_{k+1})-p_k\right].
  \end{equation*}
  However, the term in the brackets vanish because the right discrete Hamilton's equations \eqref{final} are assumed to be satisfied.
  Thus we have
  \begin{equation}
    \label{eq:p-dS+}
    q_k = D\mathcal{S}_d^k(p_k).
  \end{equation}
  Substituting this into \eqref{eq:diffS_d+} gives \eqref{eq:dHJ}.
\end{proof}

\subsection{Summary of Discrete and Continuous Results}
We have introduced the continuous and discrete variational formulations of Hamiltonian mechanics in a parallel fashion, and the correspondence between the two are summarized in Figure \ref{fig1}. Similarly, the correspondence between the continuous and discrete Type II Hamilton--Jacobi equations are summarized in Table \ref{table1}.
\tikzstyle{decision} = [ rounded corners, minimum height=4em, draw,
very thin,   text width=90pt, text  centered, node distance=2cm,
inner sep=0pt,font=\scriptsize] \tikzstyle{line} = [draw, very thin,
-stealth']
\begin{figure}[H]\caption{Continuous and discrete \textit{Type II  Hamilton's phase space variational principle}.  In the continuous case, the variation of the action functional $\mathfrak{S}$ over the space of curves gives Hamilton's equations, and the
derivatives of the extremum function $\mathcal{S}$ with respect to the boundary points yield the exact flow map of Hamilton's equation. In the
discrete case, the variation of the discrete action functional $\mathfrak{S}_d$ over the space of discrete curves
gives the discrete right Hamilton's equations and the derivatives of extremum functional $\mathcal{S}_d $ with respect to the boundary points yield the symplectic map from the initial state to the final state.}
\vspace{12pt}
    \begin{tikzpicture}\label{fig1}[node distance = 1cm, auto]{scale=0.2}
    \node [decision,] (cpq){ Cotangent Space \\  \vspace{6pt}  $(q, p) \in T^* Q $};
    \node [decision, right of =cpq, node distance=8cm] (dcpq){Disc. Cotangent Space  \\  \vspace{6pt} $(q_k, p_{k+1}) \in Q\times Q^*  $};
    \node[decision, below of=cpq] (ham) {Hamiltonian      \\   \vspace{6pt} $H(q, p)$};
    \node[decision,  below of=dcpq,  ] (dham) {Disc. Right Hamiltonian     \\  \vspace{6pt} $H_d^+(q_k, p_{k+1})$};
    \node[decision,   below of=ham,xshift=-20mm ](hact){Action Functional $\mathfrak{S}$  };
    \node[decision,  below of=dham,xshift=-20mm](dhact){Disc. Action  Functional  $\mathfrak{S}_d$};
     \node[decision,  below of=ham, xshift=20mm ](hexe){ Extremum Function $\mathcal{S}$ };
     \node[decision,  below of=dham, xshift=20mm ](dhexe){ Disc. Extremum Function $\mathcal{S}_d$ };
     \node[decision,  below of=hact,](heq){Hamilton's Eqn.  \\  \vspace{6pt} $\dot{q}=\frac{\partial H}{\partial p}, \dot{p}=-\frac{\partial H}{\partial q}$ };
    \node[decision,  below of=dhact](dheq){Disc. Right Hamilton's Eqn. \\   $q_{k}=D_{2} H_d^+(q_{k-1}, p_{k})$ \\ $ p_k= D_{1} H_d^+(q_k, p_{k+1})$ };
\node[decision, below of=hexe](bv){ $q_T=D_2 \mathcal{S}(q_0,
p_T)$
 \\ \vspace{6pt} $p_0=D_1 \mathcal{S}(q_0, p_T)
$
};
    \node[decision,  below of=dhexe](dbv){$ q_N=D_{2}
\mathcal{S}_d(q_0, p_N)  $ \vspace{3pt}  \\  \vspace{6pt}$p_0=D_{1}\mathcal{S}_d(q_0, p_N)$ };
     \node[decision, below of=heq,xshift=20mm ](csym){Symplecticity \\  $0=dd\mathcal{S}=dp_0 \wedge dq_0-dp_T \wedge dq_T$};
     \node[decision, below of=dheq,xshift=20mm](dcsym){Symplecticity  \\  $0=dd\mathcal{S}_d=dp_0 \wedge dq_0 -dp_N \wedge dq_N$};
  \path [line] (cpq)--(ham); \path[line] (ham)--(hact);
  \path[line] (ham)--(hexe);
    \path[line] (hact)--(heq); \path[line] (hexe)--(bv);
    \path[line](heq)--(csym); \path[line](bv)--(csym);

    \path [line] (dcpq)--(dham); \path[line] (dham)--(dhact); \path[line] (dham)--(dhexe);
    \path[line] (dhact)--(dheq); \path[line] (dhexe)--(dbv);
    \path[line](dheq)--(dcsym); \path[line](dbv)--(dcsym);
  \end{tikzpicture}
  \end{figure}
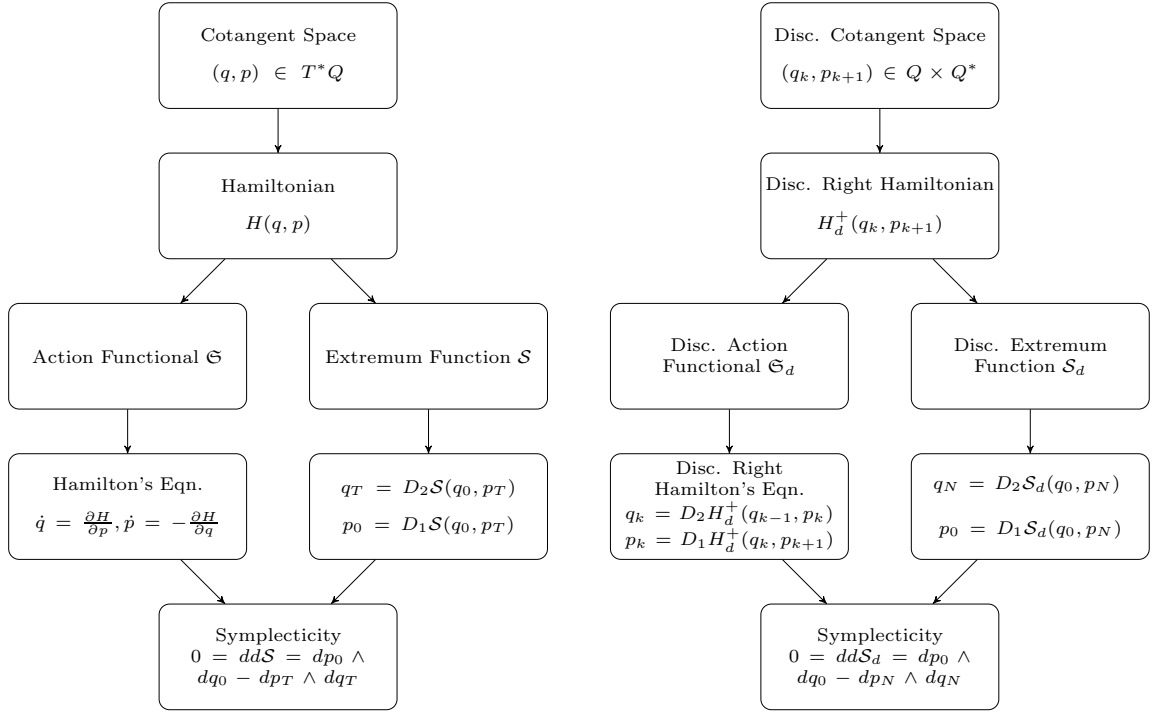

  \begin{table}[h]
    \centering
    \caption{Correspondence between ingredients in the continuous and discrete Type II Hamilton--Jacobi theories; $\mathbb{N}_{0}$ is the set of non-negative integers and $\mathbb{R}_{\geq 0}$ is the set of non-negative real numbers.}
    \label{table1}
    \renewcommand{\arraystretch}{1.5}
    \small
    \begin{tabular}{|c|c|}
      \hline
      {\bf Continuous} & {\bf Discrete}
      \\\hline\hline
      $\bigstrut (p, t) \in Q^* \times \mathbb{R}_{\geq 0} $ & $(p_k, k) \in Q^* \times \mathbb{N}_{0}$
      \\\hline
      $\displaystyle \bigstrut \dot{q} = \partial H/\partial p, $ & $\displaystyle q_{k+1} = D_{2}H_{d}^{+}(q_{k}, p_{k+1})$,
      \\
      $\displaystyle \bigstrut \dot{p} = -\partial H/\partial q $ & $\displaystyle p_{k} = D_{1}H_{d}^{+}(q_{k}, p_{k+1})$
      \\\hline
      \multirow{2}{*}{$\displaystyle \bigstrut S_2(p, t) \equiv p(t) q(t)-\int_{0}^{t} \left[ p(s) \dot{q}(s) - H(q(s), p(s))\right] ds$}
      &
      \multirow{2}{*}{$\displaystyle \bigstrut \mathcal{S}_{d}^{k}(p_{k}) \equiv p_k q_k-\sum_{l=0}^{k-1} \left[ p_{l+1}  q_{l+1} - H_{d}^{+}(q_{l}, p_{l+1})\right]$}
      \\
      &
      \\\hline
       \multirow{2}{*}{$\displaystyle dS_2 = \frac{\partial S_2}{\partial p}\,dp + \frac{\partial S_2}{\partial t}\,dt$}
      &
       \multirow{2}{*}{$\displaystyle \mathcal{S}_{d}^{k+1}(p_{k+1}) - \mathcal{S}_{d}^{k}(p_{k})$}
      \\
      &
      \\\hline
      $\displaystyle \bigstrut q\,dp + H(q,p)\,dt $ & $\displaystyle \bigstrut H_{d}^{+}(q_{k}, p_{k+1})-p_{k} \cdot q_{k}$
      \\\hline
      \multirow{2}{*}{$\displaystyle \bigstrut \frac{\partial S_2}{\partial t} = H\left( \frac{\partial S_2}{\partial p}, p \right)$}
      &
      $\displaystyle \bigstrut \mathcal{S}_{d}^{k+1}(p_{k+1}) - \mathcal{S}_{d}^{k}(p_{k})\hfill$
      \\
      &
      \hfill$=H_d^+(D\mathcal{S}_d^k(p_k), p_{k+1})-p_k\cdot D\mathcal{S}_d^k(p_k)$
      \\\hline
    \end{tabular}
  \end{table}

\section{Galerkin Hamiltonian Variational Integrators}\label{sec:galerkin_vi}

\subsection{Exact Discrete Hamiltonian} The exact discrete Hamiltonian $ H_{d,\rm exact}^+(q_0, p_{1})$ and the discrete right Hamilton's equations \eqref{final} generate a discrete solution curve $\{q_k, p_k\}_{k=0}^N$ that samples the exact solution $(q(\cdot),p(\cdot))$ of the continuous Hamilton's equations for the continuous Hamiltonian $H(q,p)$, i.e., $q_k=q(t_k)$ and $p_k=p(t_k)$.

By comparing the definition \eqref{dishd} of a discrete right Hamiltonian function $H_d^+(q_0, p_1)$ on $[0, h]$ and the corresponding discrete Hamiltonian flow in \eqref{final} to the definition  \eqref{2map} of the extremal function  on $[0, T]$ and corresponding symplectic map given by \eqref{defp0q1}, and applying Theorem \ref{thm1}, it is clear that the discrete right Hamiltonian function on $[0,h]$, given by
\begin{align}\label{exactdishd1}
 H_{d,\rm exact}^+(q_0,
 p_{1})
=  \ext_{\substack{(q, p) \in
C^2([0, T],T^*Q)\\q(t_0)=q_0, p(t_1)=p_1}} p_1 q_1 - \int_0^h \left[ p(t) \dot{q}(t)-H(q(t), p(t)) \right]dt,
\end{align}
is an exact discrete right Hamiltonian function on $[0, h]$.

\subsection{Galerkin Discrete Hamiltonian}
In general, the exact discrete Hamiltonian is not computable, since it requires one to evaluate the functional $\mathfrak{S}(q(\cdot),p(\cdot))$ given in \eqref{Htype2bv} on a solution curve of Hamilton's equations that satisfies the given boundary conditions $(q_0,p_1)$. However, the variational characterization of the exact discrete Hamiltonian naturally leads to computable approximations based on Galerkin techniques. In practice, one replaces the path space $C^2([0, T],T^*Q)$, which is an infinite-dimensional function space, with a finite-dimensional function space, and uses numerical quadrature to approximate the integral. 

Let $\{\psi_i(\tau)\}_{i=1}^s$, $\tau \in [0, 1]$, be a set of basis functions for a $s$-dimensional function space $C_d^s$. We also choose a numerical quadrature formula with quadrature weights $b_i$, and quadrature points $c_i$, i.e., $\int_0^1 f(x) dx \approx \sum_{j=1}^s b_i f(c_i)$. From these basis functions and the numerical quadrature formula, we will systematically construct a generalized Galerkin Hamiltonian variational integrator in the following manner:

\renewcommand{\theenumi}{\arabic{enumi}.}
\renewcommand{\labelenumi}{\theenumi}
\renewcommand{\theenumii}{\roman{enumii}.}
\renewcommand{\labelenumii}{\theenumii}

\begin{enumerate}
  \item Use the basis functions $\psi_i$ to approximate the velocity $\dot{q}$ over the interval $[0,h]$,
\[
  \dot{q}_d(\tau h)=\sum_{i=1}^s V^i \psi_i(\tau).
\]
  \item Integrate $\dot{q}_d(\rho)$ over $[0,
\tau h],$ to obtain the approximation for the position $q$,
\[
q_d(\tau h)=q_d(0)+ \int_0^{\tau h }\sum_{i=1}^s V^i \psi_i(\rho) d(\rho h)=q_0+h\sum_{i=1}^s V^i \int_0^{\tau} \psi_i(\rho) d\rho,
\]
where we applied the boundary condition $q_d(0)=q_0$. Applying the boundary condition $q_d(h)=q_1$ yields
\[
 q_1=q_d(h)=q_0+h\sum_{i=1}^s V^i \int_0^1 \psi_i(\rho) d\rho\equiv q_0+h\sum_{i=1}^s B_i V^i,
\]
where $B_i=\int_0^{1} \psi_i(\tau)d\tau$.
Furthermore, we introduce the internal stages,
\[
 Q^i\equiv q_d(c_i h)=q_0+h\sum_{j=1}^s V^j
\int_0^{c_i} \psi_j(\tau) d\tau \equiv q_0+h\sum_{j=1}^s A_{ij} V^j,
\]
where $A_{ij}=\int_0^{c_i} \psi_j(\tau) d\tau$.

\item Let $P^i=p(c_i h)$. Use the numerical quadrature formula $(b_i,c_i)$ and the finite-dimensional function space $C_d^s$ to construct $H_d^+(q_0, p_1)$ as follows
\begin{align}\nonumber
H_d^+(q_0, p_1) &\approx   \ext_{\substack{(q, p) \in
C^2([0, T],T^*Q)\\q(t_0)=q_0, p(t_1)=p_1}} p_1 q_1 - \int_0^h \left[ p(t) \dot{q}(t)-H(q(t), p(t)) \right]dt
\\ \label{H2d}
H_d^+(q_0, p_1) &=\ext_{q_d\in C_d^s([0,h],Q), P_i\in Q^*}  \left\{p_1 q_d(h)-h\sum_{i=1}^s
b_i\left[p(c_i h) \dot{q}_d(c_i h)  - H(q_d(c_i h), p(c_i h))\right]\right\}
\\
\nonumber &=\ext_{V^i, P^i}\left\{ p_1\left(
q_0+h\sum_{i=1}^s B_i V^i\right)-h\sum_{i=1}^s b_i\left[P^i \sum_{j=1}^s
V^j \psi_j(c_i) - H\left( q_0+h\sum_{j=1}^s A_{ij} V^j, P^i\right)\right]\right \} \\
\nonumber & \equiv \ext_{V^i,P^i}  K(q_0, V^i, P^i, p_1).
\end{align}
To obtain an expression for $H_d^+(q_0,p_1)$, we first compute the stationarity conditions for $K(q_0,V^i, P^i,p_1)$ under the fixed boundary condition $(q_0,p_1)$.
\begin{subequations}\label{gsta}
\begin{alignat}{2}\label{gsta1}
0&=\frac{\partial K(q_0, V^i, P^i, p_1)}{\partial V^j} = h  p_1 B_j
-h\sum_{i=1}^{s}b_i\left(P^i \psi_j(c_i) - h A_{ij} \frac{\partial H}{\partial q}(Q^i,P^i)
   \right),&\qquad j&=1, \ldots, s.\\
\label{gsta2}
0&=\frac{\partial  K(q_0, V^i, P^i, p_1)}{\partial P^j} =h b_j
\left(\sum_{i=1}^s \psi_i(c_j )V^i - \frac{\partial H}{\partial p}( Q^j, P^j)\right),& j&=1, \ldots, s.
 \end{alignat}
\end{subequations}

\item By solving the $2 s$ stationarity conditions \eqref{gsta}, we can express the parameters $V^i,
P^i$, in terms of $q_0, p_1$, i.e., $ V^i=V^i(q_0, p_1)$ and
$P^i=P^i(q_0, p_1).$ Then, the symplectic map $(q_0,p_0)\mapsto (q_1,p_1)$ can be expressed in terms of the internal stages
\begin{align}\label{gdefp0}
p_0&=\frac{\partial H_d^+(q_0, p_1)}{\partial q_0} =
\frac{\partial K(q_0, V^i(q_0, p_1), P^i(q_0, p_1), p_1)}{\partial
q_0} \\ \nonumber &= \frac{\partial K}{\partial q_0} +\frac{\partial
K}{\partial V^i}\frac{\partial V^i}{\partial q_0}+\frac{\partial
K}{\partial P^i}\frac{\partial
P^i}{\partial q_0} =\frac{\partial K}{\partial q_0}\\
\nonumber &= p_1+ h\sum_{i=1}^s b_i \frac{\partial H}{\partial q}( Q^i,P^i),
\end{align}
Similarly, we obtain
\begin{align} \label{gdefq1}
q_1&=\frac{\partial H_d^+(q_0, p_1)}{\partial p_1} = \frac{\partial K(q_0, V^i(q_0, p_1), P^i(q_0, p_1),
p_1)}{\partial p_1} \\ \nonumber &= \frac{\partial K}{\partial
V^i}\frac{\partial V^i}{\partial p_1}+\frac{\partial K}{\partial
P^i}\frac{\partial P^i}{\partial p_1}+\frac{\partial
K}{\partial p_1} =\frac{\partial K}{\partial p_1}\\
\nonumber &= q_0+ h\sum_{i=1}^s B_i V^i.
\end{align}
\end{enumerate}

Without loss of generality, we assume that the quadrature weights $b_i \neq 0$. Then, the stationarity condition
\eqref{gsta2} reduces to $\sum_{i=1}^s \psi_i(c_j)V^i - \frac{\partial H}{\partial p}(
Q^j,P^j)=0.$ Moreover, by substituting \eqref{gdefp0} into the
stationarity condition \eqref{gsta1}, we obtain,
$\sum_{i=1}^s b_i P^i \psi_j(c_i)-p_0 B_j+h\sum_{i=1}^s(b_i B_j -b_i
A_{ij}) \frac{\partial H}{\partial q}(Q^i, P^i)=0$.

In summary, the above procedure gives a systematic way to construct a generalized Galerkin Hamiltonian
variational integrator, which can be rewritten in the following compact form,
\begin{subequations}
\begin{alignat}{2}\label{sprkp1}
q_1&=q_0+h \sum_{i=1}^s B_i V^i,\\
\label{sprkp2}
p_1&=p_0-h\sum \limits_{i=1}^s b_i \frac{\partial H}{\partial q}(Q^i,P^i),\\
\label{sprkp3}
    Q^i&=q_0+h \sum \limits_{j=1}^s A_{ij} V^j,& i&=1, \ldots, s,\\
\label{sprkp4}
0&=\sum_{i=1}^s b_i P^i \psi_j(c_i)-p_0 B_j+h\sum_{i=1}^s(b_i B_j -b_i
A_{ij}) \frac{\partial H}{\partial q}(Q^i, P^i),&\qquad j&=1, \ldots, s,\\
\label{sprkp5}
0&=\sum_{i=1}^s \psi_i(c_j)V^i - \frac{\partial H}{\partial p}( Q^j,P^j),& j&=1, \ldots, s,
   \end{alignat}
\end{subequations}
where $(b_i, c_i)$ are the quadrature weights and quadrature points, and $B_i=\int_0^1 \psi_i(\tau) d\tau$,  $A_{ij}= \int_0^{c_i}
\psi_j(\tau) d\tau$.

This is the general form of a Galerkin Hamiltonian variational integrator. Issues of solvability, convergence, and accuracy, depend on the specific Hamiltonian system, and the choice of finite-dimensional function space $C_d^s$ and numerical quadrature formula $(b_i, c_i)$. We will not perform an in depth analysis here, but we will illustrate how our proposed framework is related to the discrete Lagrangian based methods given in \cite{MaWe2001} and p.\,209 of \cite{HaLuWa2006}.

\subsection{Galerkin Variational Integrators from the Lagrangian Point of View}

In this subsection, we investigate the generalized Galerkin variational
integrators from Lagrangian point of view when the Hamiltonian
function is hyperregular. In this case, the exact discrete right
Hamiltonian function is related by the Legendre transformation to exact discrete Lagrangian
function, i.e.,
\begin{align}
L_d^{\rm exact}(q_0, q_1)&= \ext_{\substack{q\in C^2([0,h],Q) \\ q(0)=q_0, q(h)=q_1}} \int_0^h L(q(t), \dot q(t)) dt\\
\nonumber &=  \ext_{\substack{(q,p)\in C^2([0,h],T^*Q)\\q(0)=q_0, q(h)=q_1}} \int_0^h p\dot{q} -H(q,p) dt\\
\nonumber &=\left.p_1 q_1 - H_{d,\rm exact}^+(q_0, p_1)\right|_{q_1=D_2 H_{d,\rm exact}^+(q_0, p_1)}.
\end{align}

We wish to see how Galerkin variational integrators that are derived from the Hamiltonian and Lagrangian sides are related. In order for the comparison to make sense, we will approximate the exact discrete Lagrangian using the same basis functions and numerical quadrature formula as on the Hamiltonian side. As before, let $\{\psi_i(\tau)\}_{i=1}^s$, $\tau \in [0, 1]$, be a set of basis functions for a $s$-dimensional function space $C_d^s$, and choose a numerical quadrature formula with quadrature weights $b_i$, and quadrature points $c_i$. From these basis functions and the numerical quadrature formula, we will systematically construct a generalized Galerkin Lagrangian variational integrator in the following manner:

\begin{enumerate}
  \item Use the basis functions $\psi_i$ to approximate the velocity $\dot{q}$ over the interval $[0,h]$,
\[
  \dot{q}_d(\tau h)=\sum_{i=1}^s V^i \psi_i(\tau).
\]
  \item Integrate $\dot{q}_d(\rho)$ over $[0,
\tau h],$ to obtain the approximation for the position $q$,
\[
q_d(\tau h)=q_d(0)+ \int_0^{\tau h }\sum_{i=1}^s V^i \psi_i(\rho) d(\rho h)=q_0+h\sum_{i=1}^s V^i \int_0^{\tau} \psi_i(\rho) d\rho,
\]
where we applied the boundary condition $q_d(0)=q_0$. In the discrete Lagrangian framework, the boundary conditions are given by $(q_0, q_1)$, so we will use a Lagrange multiplier to enforce the boundary condition $q_d(h)=q_1$,
\[
 q_1=q_d(h)=q_0+h\sum_{i=1}^s V^i \int_0^{1} \psi_i(\rho) d\rho\equiv q_0+h\sum_{i=1}^s B_i V^i,
\]
where $B_i=\int_0^{1} \psi_i(\tau)d\tau$.
Furthermore, we introduce the internal stages,
\begin{equation}\label{lgdefqi}
 Q^i\equiv q_d(c_i h)=q_0+h\sum_{j=1}^s V^j
\int_0^{c_i} \psi_j(\tau) d\tau \equiv q_0+h\sum_{j=1}^s A_{ij} V^j,
\end{equation}
where $A_{ij}=\int_0^{c_i} \psi_j(\tau) d\tau$, and their velocities,
\begin{align}
\dot{Q}^i&\equiv\dot{q}_d(c_i h)=\sum_{j=1}^s \psi_j(c_i)V^j.
\end{align}
\item Use the numerical quadrature formula $(b_i, c_i)$, and the finite-dimensional function space $C_d^s$ to construct $L_d(q_0,q_1)$ as follows
\begin{align}
\nonumber L_d(q_0, q_1) &\approx \ext_{\substack{q\in C^2([0,h],Q),\lambda \\ q(0)=q_0}} \left[\int_0^h L(q(t), \dot q(t)) dt\right] + \lambda
(q_1-q(h))\\
\label{Ld} L_d(q_0, q_1) &=\ext_{\substack{q_d \in C_d^s([0, h],Q), \lambda\\q(0)=q_0}}
\left[h\sum_{i=1}^s b_i L(q_d(c_i h), \dot{q}_d(c_i h) )\right]+ \lambda
\left(q_1-q_d(h)\right)\\
\nonumber &=\ext_{V^i, \lambda}
\left\{\left[h\sum_{i=1}^s b_i L\left(q_0+h\sum_{j=1}^s A_{ij} V^j, \sum_{j=1}^s V^j \psi_j(c_i)\right)\right]+ \lambda
\left(q_1-q_0-h\sum_{i=1}^s B_i V^i\right)\right\}\\
\nonumber & \equiv \ext_{V^i, \lambda} K(q_0, V^i, \lambda, q_1).
\end{align}
To obtain an expression for $L_d(q_0,q_1)$, we first compute the stationarity conditions for $K(q_0, V^i, \lambda, q_1)$ under the fixed boundary condition $(q_0, q_1)$.
\begin{subequations}\label{lgsta}
\begin{alignat}{2}
\label{lgsta1}
0&=\frac{\partial K(q_0, V^i, \lambda, q_1)}{\partial V^j} = h\sum_{i=1}^{s}b_i\left(\frac{\partial L}{\partial q}(Q^i, \dot{Q}^i) h A_{ij}+
\frac{\partial L}{\partial \dot{q}}(Q^i, \dot{Q}^i) \psi_j(c_i)\right)-h\lambda B_j,&\qquad j&=1,
 \ldots, s.\\
\label{lgsta2}
0&=\frac{\partial  K(q_0, V^i, \lambda, q_1)}{\partial \lambda}
=q_1-q_0-h\sum_{i=1}^s B_i V^i.
 \end{alignat}
\end{subequations}
\item By solving the $2s$ stationarity equations \eqref{lgsta}, we can express the parameters $V^i, \lambda$, in terms of $q_0, q_1$, i.e., $V^i=V^i(q_0,
q_1),\lambda=\lambda(q_0, q_1)$. Then, the symplectic map $(q_0, p_0)\mapsto (q_1, p_1)$ can be expressed in terms of the internal stages and the Lagrange multiplier
\begin{align}\label{lgdefp0}
p_0&=-\frac{\partial L_d(q_0, q_1)}{\partial q_0} =-\left( \frac{\partial K(q_0, V^i(q_0, p_1), \lambda(q_0, p_1),
q_1)}{\partial q_0}\right)
\\ \nonumber &= -\left(\frac{\partial K}{\partial q_0} +\frac{\partial
K}{\partial V^i}\frac{\partial V^i}{\partial q_0}+\frac{\partial
K}{\partial \lambda}\frac{\partial \lambda}{\partial q_0}\right) =-\frac{\partial K}{\partial q_0}\\
\nonumber &= -h\sum_{i=1}^s b_i L_q( Q^i,\dot{Q}^i)+\lambda.
\end{align}
Similarly, we obtain
\begin{align} \label{lgdefp1}
p_1&=\frac{\partial L_d(q_0, q_1)}{\partial q_1} =
\frac{\partial K(q_0, V^i(q_0, p_1), \lambda(q_0, p_1),
q_1)}{\partial q_1}
\\ \nonumber &= \frac{\partial K}{\partial q_1} +\frac{\partial
K}{\partial V^i}\frac{\partial V^i}{\partial q_1}++\frac{\partial
K}{\partial \lambda}\frac{\partial \lambda}{\partial q_1} =\frac{\partial K}{\partial q_1}\\
\nonumber &= \lambda.
\end{align}
\end{enumerate}

By combining \eqref{lgdefp0} and \eqref{lgdefp1}, we obtain $p_1=p_0+h\sum_{i=1}^s b_i L_q(Q^i,\dot{Q}^i)$. Substituting this into the stationarity condition \eqref{lgsta1}, yields $\sum_{i=1}^s b_i \partial L/\partial \dot{q}(Q^i, \dot{Q}^i) \psi_j(c_i)-p_0 B_j-h\sum_{i=1}^s(b_i B_j -b_i A_{ij}) \partial L/\partial q(Q^i, \dot{Q}^i)=0$.

In summary, the above procedure gives a systematic way to construct a generalized Galerkin Lagrangian variational integrator, which can be written in the following compact form,
\begin{subequations}
\begin{alignat}{2}\label{lsprkp1}
q_1&=q_0+h \sum_{i=1}^s B_i V^i,\\
\label{lsprkp2}
p_1&=p_0+h\sum \limits_{i=1}^s b_i \frac{\partial L}{\partial q}(Q^i,\dot{Q}^i),\\
\label{lsprkp3}
Q^i&=q_0+h \sum \limits_{j=1}^s A_{ij} V^j,&\qquad i&=1, \ldots, s\\
\label{lsprkp4}
0&=\sum_{i=1}^s b_i \frac{\partial L}{\partial \dot{q}}(Q^i, \dot{Q}^i) \psi_j(c_i)-p_0 B_j-h\sum_{i=1}^s(b_i B_j -b_i
A_{ij}) \frac{\partial L}{\partial q}(Q^i, \dot{Q}^i),&\qquad j&=1, \ldots, s\\
\label{lsprkp5}
0&=\sum_{i=1}^s \psi_i(c_j)V^i- \dot{Q}^j,&\qquad j&=1, \ldots, s
\end{alignat}
\end{subequations}
where $(b_i, c_i)$ are the quadrature weights and quadrature points, and $B_i=\int_0^1 \psi_i(\tau) d\tau$,  $A_{ij}= \int_0^{c_i}
\psi_j(\tau) d\tau$.

As expected, this is equivalent to the generalized Galerkin Hamiltonian variational integrator, as the following proposition indicates.

\begin{proposition} \label{prop:equivalence} If the continuous Hamiltonian $H(q,p)$ is hyperregular, and we construct a Lagrangian $L(q,\dot{q})$ by the Legendre transformation, then the generalized Galerkin Hamiltonian variational integrator \eqref{sprkp1}--\eqref{sprkp5} and the generalized Galerkin Lagrangian variational integrator \eqref{lsprkp1}--\eqref{lsprkp5}, associated with the same choice of basis functions $\psi^i$ and numerical quadrature formula $(b_i, c_i)$, are equivalent.
\end{proposition}
\begin{proof}
Since we chose the same basis functions and numerical quadrature formula for both methods, the approximations for $q_1$ and $Q^i$ are the same in both methods, as can be seen by comparing  \eqref{sprkp1} and \eqref{lsprkp1}, \eqref{sprkp3} and \eqref{lsprkp3}. Since we assumed that the Lagrangian and Hamiltonian are related by $L(q,\dot q)=p\dot q - H(q,p)$, subject to the Legendre transformation $\dot{q}=\partial H/\partial p(q,p)$, we consider $p$ to be a function of $(q,\dot{q})$, and compute
\begin{align*}
\frac{\partial L}{\partial q}(q,\dot{q})&=\dot{q}\cdot \frac{\partial p}{\partial q}-\frac{\partial H}{\partial q}(q,p)-\frac{\partial H}{\partial p}(q,p)\frac{\partial p}{\partial q}=-\frac{\partial H}{\partial q}(q,p),\\
\frac{\partial L}{\partial \dot{q}}(q,\dot{q})&= \dot{q}\cdot\frac{\partial p}{\partial \dot{q}}+p-\frac{\partial H}{\partial p}(q,p)\frac{\partial p}{\partial \dot{q}}=p.
\end{align*}
Since these identities have to hold at the internal stages, we have that
\begin{align*}
\frac{\partial H}{\partial p}(Q^i, P^i)&=\dot{Q}^i,\\
\frac{\partial H}{\partial q}(Q^i,P^i)&=\frac{\partial L}{\partial q}(Q^i,\dot{Q}^i),\\
P^i&=\frac{\partial L}{\partial \dot{q}}(Q^i,\dot{Q}^i),
\end{align*} 
for $i=1,\ldots,s$. Clearly, substituting these identities into \eqref{sprkp2}, \eqref{sprkp4}, and \eqref{sprkp5}, yields \eqref{lsprkp2}, \eqref{lsprkp4}, and \eqref{lsprkp5}. As such the two systems of equations, \eqref{sprkp1}--\eqref{sprkp5} and \eqref{lsprkp1}--\eqref{lsprkp5}, are equivalent, once the Legendre transformation and the identities relating the continuous Lagrangian and Hamiltonian are taken into account.
\end{proof}

%To investigate the relation of generalized variational integrators
%from Hamiltonian point of view, we would like to use the same basis
%function and quadrature formula as on the Hamiltonian side. Then we
%can see that
% is similar to the construction

\subsection{Variational integrators and Symplectic Partitioned Runge-Kutta methods}
In this subsection, we consider a special class of Galerkin variational integrators, and demonstrate that they can be implemented as symplectic partitioned Runge--Kutta methods.

Let $C_d^s([0,1],Q)$ be a $s$-dimensional function space, and consider a set of basis functions $\psi_i(\tau)$ on $[0,1]$, and a set of control points $c_i$, $i=1,\ldots, s$. We would like to construct a new set of basis functions $\phi_i(\tau)$ that span the same function space, and satisfies $\phi_i(c_j)=\delta_{ij},$ where $\delta_{ij}$ is the Kronecker delta. This is possible whenever the matrix
\begin{align}\label{invertible} M=
\begin{bmatrix}
      \psi_1(c_1) & \psi_1(c_2) &\cdots & \psi_1(c_s) \\
      \psi_2(c_1) & \psi_2(c_2) &\cdots & \psi_2(c_s)\\
      \vdots & \vdots & \ddots & \vdots\\
     \psi_s(c_1) & \psi_s(c_2) &\cdots & \psi_s(c_s)\\
\end{bmatrix}
\end{align}
is invertible. In particular, let $\psi(\cdot)=[\psi_1(\cdot), \ldots,
\psi_s(\cdot)]^T,$ and construct a new set of basis functions $\phi(\cdot)=[\phi_1(\cdot),\ldots,
\phi_s(\cdot)]^T$ by $\phi(\cdot)=M^{-1} \psi(\cdot).$  It is easy
to see that $\phi_i(c_j )=\delta_{ij}$ since
\begin{align}
\begin{bmatrix}
      \phi_1(c_1) & \phi_1(c_2) &\cdots & \phi_1(c_s) \\
      \phi_2(c_1) & \phi_2(c_2) &\cdots & \phi_2(c_s)\\
      \vdots & \vdots & \ddots & \vdots\\
     \phi_s(c_1) & \phi_s(c_2) &\cdots & \phi_s(c_s)
\end{bmatrix}
&=M^{-1}M=
\begin{bmatrix}
      1 & 0 &\cdots & 0 \\
      0 & 1 &\cdots & 0\\
      \vdots & \vdots & \ddots & \vdots\\
    0 & 0 &\cdots & 1\\
\end{bmatrix}.
\end{align}

We can construct a numerical quadrature formula that is exact on the span of the basis functions $\psi_i(\tau)$ as follows: Since $\phi_i(c_j)=\delta_{ij}$, we can interpolate any function $f(\tau)$ on $[0,1]$ at the control points $c_i$ by taking $\bar{f}(\tau)=\sum_{i=1}^s f(c_i)\phi_i(\tau)$. Then, we obtain the following quadrature formula,
\begin{align}\label{quad}
\int_0^1 f(\tau)d\tau&\approx \int_0^1 \bar{f}(\tau)d\tau=\int_0^1 \sum_{i=1}^s f(c_i)\phi_i(\tau)
d\tau = \sum_{i=1}^s f(c_i)  \left[\int_0^1 \phi_i(\tau)d\tau\right] \equiv \sum_{i=1}^s b_i f(c_i),
\end{align}
where $b_i=\int_0^1 \phi_i(\tau)d\tau$ are the quadrature weights. By construction, the above quadrature formula is exact for any function in the $s$-dimensional function space $C_d^s([0,1],Q)$. Now, if we apply this quadrature formula with quadrature points $c_i$, we obtain a Galerkin variational integrator that can be implemented as a symplectic partitioned Runge--Kutta (SPRK) method.

\begin{theorem}\label{thm3}
Given any set of basis functions $\psi(\cdot)=[\psi_1(\cdot),
\ldots, \psi_s(\cdot)]^T,$ that span $C_d^s([0,1],Q)$, consider the quadrature formula given in \eqref{quad}. Then, the associated generalized Galerkin Hamiltonian variational integrator \eqref{sprkp1}--\eqref{sprkp5}, which is expressed in terms of the discrete right Hamiltonian function \eqref{H2d}, can be implemented by the following $s$-stage symplectic partitioned Runge--Kutta method applied to Hamilton's equations \eqref{Hameq}:
\begin{subequations}
\label{sprkp}
\begin{alignat}{2}\label{sprkp11}
q_1&=q_0+h\sum \limits_{i=1}^s b_i \frac{\partial H}{\partial p}(Q^i, P^i),\\
\label{sprkp22}
p_1&=p_0-h\sum \limits_{i=1}^s b_i \frac{\partial H}{\partial q}(Q^i, P^i),\\
\label{sprkp33}
    Q^i&=q_0+h \sum \limits_{j=1}^s a_{ij} \frac{\partial H}{\partial p}(Q^j, P^j),&\qquad i&=1,\ldots,s,\\
\label{sprkp44}
P^i&=p_0-h \sum \limits_{j=1}^s \tilde{a}_{ij} \frac{\partial H}{\partial q}(Q^j, P^j),&\qquad i&=1,\ldots,s,
\end{alignat}
\end{subequations}
where $b_i=\int_0^1 \phi_i(\tau) d\tau \neq 0$,  $a_{ij}= \int_0^{c_i}
\phi_j(\tau) d\tau $, and
$\tilde{a}_{ij}=\frac{b_i b_j -b_j a_{ji}}{b_i}$. The basis functions satisfy $\phi_i(c_j)=\delta_{ij}$, and are given by
$\phi(\cdot)=M^{-1} \psi(\cdot)$,  where
$\phi(\cdot)=[\phi_1(\cdot), \ldots, \phi_s(\cdot)]^T$, and $M$ is
defined in \eqref{invertible}.
\end{theorem}
 \begin{proof}
The new basis functions $\phi(\tau)$ are constructed from the original basis functions $\psi(\tau)$ by the relationship $\phi(\cdot)=M^{-1}\psi(\cdot)$. Thus, we have that $\psi(\cdot)=M\phi(\cdot)$, and in particular, $\psi_i(\tau)=\sum_{j=1}^s \psi_i(c_j) \phi_j(\tau)$. By substituting this into $B_i=\int_0^1 \psi_i(\tau ) d\tau$, equation \eqref{sprkp1} becomes
\begin{align*}
q_1&=q_0+h \sum_{i=1}^s B_i V^i \\ \nonumber &=q_0+h \sum_{i=1}^s
 V^i \int_0^1 \sum_{j=1}^s \psi_i(c_j) \phi_j(\tau) d\tau  \\
\nonumber &=q_0+h \sum_{j=1}^s \sum_{i=1}^s \psi_i(c_j)V^i \int_0^1
\phi_j(\tau) d\tau \\ \nonumber &=q_0+h \sum_{j=1}^s \frac{\partial H}{\partial p}(Q^j, P^j)
\int_0^1 \phi_j(\tau) d\tau  \\ \nonumber & \equiv q_0+h
\sum_{j=1}^s b_j \frac{\partial H}{\partial p}(Q^j, P^j),
\end{align*}
where we used equation \eqref{sprkp5} to go from the third equality to the fourth one. Similarly, by substituting $\psi_k(\tau)=\sum_{j=1}^s \psi_k(c_j)
\phi_j(\tau)$ into $A_{ik}=\int_0^{c_i} \psi_k(\tau) d\tau$ and using equation \eqref{sprkp5}, equation \eqref{sprkp3} becomes
\begin{align*}
 Q^i&=q_0+h \sum \limits_{k=1}^s A_{ik} V^k
\\ \nonumber &=q_0+h \sum_{k=1}^s V^k
\int_0^{c_i}  \sum_{j=1}^s \psi_k(c_j)\phi_j(\tau) d\tau \\
\nonumber &=q_0+h
\sum_{j=1}^s \frac{\partial H}{\partial p}(Q^j, P^j) \int_0^{c_i} \phi_j(\tau) d\tau  \\
\nonumber &\equiv q_0+h \sum_{j=1}^s a_{ij} \frac{\partial H}{\partial p}(Q^j, P^j),
\end{align*}
where $a_{ij}=\int_0^{c_i} \phi_j(\tau) d\tau.$ Note that equation \eqref{sprkp2} has the same form as
\eqref{sprkp22}, so we only have to recover equation
\eqref{sprkp44}. Let $E^{\psi}=[E^{\psi}_1, \ldots, E^{\psi}_s]^T,$
where $E^{\psi}_j\equiv \sum_{i=1}^s b_i P^i\psi_j(c_i)-p_0 B_j+h\sum_{i=1}^s(b_i B_j -b_i A_{ij}) \frac{\partial H}{\partial q}(Q^i,
P^i)=0$, which corresponds to \eqref{sprkp4}. By substituting $
B_i=\int_0^1 \psi_i(\tau ) d\tau =\int_0^1 \sum_{j=1}^s \psi_i(c_j )
\phi_j(\tau) d\tau$, and $b_i=\int_0^1 \phi_i(\tau) d\tau$ into $E^{\psi}_j$, we obtain
\begin{align*}
0&=E^{\psi}_j=\sum_{i=1}^{s}b_i P^i \psi_j(c_i)-p_0 B_j +h \sum_{i=1}^{s}(b_i B_j- b_i A_{ij}) \frac{\partial H}{\partial q}(Q^i,P^i) \\
\nonumber &=\sum_{i=1}^{s}b_i P^i \psi_j(c_i)-p_0 \int_0^1 \psi_j(\tau) d\tau+ h\sum_{i=1}^{s}\left(b_i \int_0^1 \psi_j(\tau) d\tau
 - b_i \int_0^{c_i} \psi_j(\tau) d\tau\right) \frac{\partial H}{\partial q}(Q^i,P^i) \\
\nonumber &=\sum_{i=1}^{s}b_i P^i \psi_j(c_i)-p_0 \int_0^1\sum_{i=1}^s \psi_j(c_i)\phi_i(\tau) d\tau\\
\nonumber &\qquad +h\sum_{i=1}^{s}\left(b_i \int_0^1\sum_{k=1}^s \psi_j(c_k)\phi_k(\tau) d\tau - b_i \int_0^{c_i}\sum_{k=1}^s \psi_j(c_k )\phi_k(\tau) d\tau \right)\frac{\partial H}{\partial q}(Q^i,P^i) \\
\nonumber &=\sum \limits_{i=1}^s \psi_j(c_i)\left( b_i P^i -b_ip_0   + h \sum_{k=1}^{s}(b_i b_k-b_k a_{ki} )\frac{\partial H}{\partial q}(Q^k,P^k)\right).
\end{align*}
We swapped the role of the indices $i$ and $k$ in the second to last line to obtain the final equality.
Let $E^{\phi}=[E^{\phi}_1, \ldots, E^{\phi}_s]^T,$ where $E^{\phi}_i \equiv b_i P^i -b_ip_0   + h \sum_{k=1}^{s}(b_i b_k-b_k a_{ki} ) \frac{\partial H}{\partial q}(Q^k,P^k)$. Then, the above equation can be viewed as the $j$-th component of the system of equations $M E^{\phi}
\equiv E^{\psi}=0,$ where $M=[\psi_i(c_j)]$ is invertible. Therefore, we have that $E^{\phi}=0,$ i.e., $E^{\phi}_i=b_i P^i -b_ip_0   + h
\sum_{k=1}^{s}(b_i b_k-b_k a_{ki} ) \frac{\partial H}{\partial q}(Q^k,P^k)=0$. Since $b_i \neq
0$, dividing by $b_i$ and recalling that $\tilde{a}_{ij}=\frac{b_i b_j -b_j a_{ji}}{b_i}$ yields \eqref{sprkp44}.
\end{proof}
\paragraph{\bf Comparison with Discrete Lagrangian SPRK Methods} Proposition \ref{prop:equivalence} states that for hyperregular Hamiltonians, if one chooses the same basis functions and quadrature formula, the generalized Galerkin Hamiltonian variational integrator is equivalent to the generalized Galerkin Lagrangian variational integrator. Therefore, the above theorem also applies in the Lagrangian setting. In particular, if one chooses the Lagrange polynomials associated with the quadrature nodes $c_i$ as our choice of basis functions $\psi_i(\tau)$, then the coefficients of the SPRK method derived above agree with the method derived in \cite{MaWe2001} using discrete Lagrangians. However, our approach remains valid in the case of degenerate Hamiltonians, for which it is impossible to obtain a Lagrangian and apply the method in \cite{MaWe2001} to derive Hamiltonian variational integrators.

The derivation on p.\,209 of the book \cite{HaLuWa2006}, which is analogous to the result in \cite{Su1990}, generalizes  the approach in \cite{MaWe2001} by considering discrete Lagrangian SPRK methods without the restriction that the Runge--Kutta coefficients are obtained from integrals of Lagrange polynomials. It is however unclear how one should choose these coefficients. In contrast, our approach provides a systematic means of deriving the coefficients by an appropriate choice of basis functions and quadrature formula. Our discrete Hamiltonian method is expressed in terms of Type II generating functions and the continuous Hamiltonian as opposed to the discrete Lagrangian approach based on Type I generating functions and the continuous Lagrangian.\\

\paragraph{\bf Discrete Hamiltonian associated with Galerkin SPRK Method.} For the symplectic partitioned Runge--Kutta method \eqref{sprkp} described above, we can explicitly compute the corresponding Type II generating function $H_d^+(q_0,p_1)$ given in \eqref{H2d} as follows,
\begin{align}
H_d^+(q_0, p_1)&= p_1 q_d(h)-h\sum_{i=1}^s
b_i\left[P^i \dot{q}_d(c_i h)  - H(Q^i, P^i)\right] \\
\nonumber &=p_1 \left(q_0+h\sum_{i=1}^s b_i \frac{\partial H}{\partial p}(Q^i, P^i)\right)-h\sum_{i=1}^s
b_i\left[P^i \frac{\partial H}{\partial p}(Q^i, P^i) - H(Q^i, P^i)\right] \\ \nonumber &= p_1
q_0+h\sum_{i=1}^s b_i (p_1-P^i)\frac{\partial H}{\partial p}(Q^i, P^i)+h\sum_{i=1}^s b_i
H(Q^i, P^i) \\ \nonumber &=p_1q_0 -h^2 \sum_{i, j=1}^s b_i a_{ij}
\frac{\partial H}{\partial q}(Q^i, P^i) \frac{\partial H}{\partial p}(Q^j, P^j)+h\sum_{i=1}^s b_i H(Q^i, P^i).
\end{align}
This Type II generating function is consistent with the Type I generating function for SPRK methods that was given in Theorem 5.4 on p.\,198 of \cite{HaLuWa2006}.\\

\paragraph{\bf Sufficient Condition for Consistency of the SPRK Method.} If the constant function $f(x)=1$ is in the finite-dimensional function space $C_d^s$, then by interpolation, we have that $1=\sum_{i=1}^s f(c_i) \phi_i(\tau)=\sum_{i=1}^s \phi_i(\tau)$. Thus, $\sum_{i=1}^s b_i=\int_0^1
\sum_{i=1}^s \phi_i(\tau) d\tau=1$. Partitioned Runge--Kutta order theory~\cite{Bu2008} states that the condition $\sum_{i=1}^s b_i=1$ implies that the variational integrator \eqref{sprkp} is at
least first-order. Therefore, to obtain a consistent method, it is sufficient that the constant function is in the span of the basis functions we choose. In particular, if we let $\psi_1(\tau)=1$, we ensure that our method is at least first-order.\\

\paragraph{\bf Construction of the SPRK Tableau.}
Let the symplectic partitioned Runge--Kutta method \eqref{sprkp} be denoted by the tableau.

\centerline{
\begin{tabular}{c|c c c}
   $c_1$ & $a_{11}$ & $\cdots$ & $a_{1s}$ \\
$ \vdots$ & $\vdots$ & & $\vdots$\\
$c_s$ & $a_{s1} $& $\cdots$ & $a_{ss}$\\
 \hline
& $ b_1$ & $\cdots$ &$b_s $\\
 \end{tabular}
\qquad\quad
\begin{tabular}{c|c c c}
   $\tilde{c}_1$ & $\tilde{a}_{11}$ & $\cdots$ & $\tilde{a}_{1s}$ \\
$ \vdots$ & $\vdots$ & & $\vdots$\\
$\tilde{c}_s$ & $\tilde{a}_{s1} $& $\cdots$ & $\tilde{a}_{ss}$\\
 \hline
& $ b_1$ & $\cdots$ &$b_s $\\
 \end{tabular}
}
\vspace*{2ex}

Based on the above generalized Galerkin method, the coefficients in the partitioned Runge--Kutta tableau can be constructed in the following systematic way:

\begin{boxedminipage} [t]{15cm}
{\bf Generalized Galerkin Hamiltonian SPRK Method.}
\begin{enumerate}
  \item Choose a basis set $\psi_i(\tau), \tau \in [0,1], i=1, \ldots, s$, with
  $\psi_1(\tau)=1.$
  \item Choose quadrature points $c_i, i=1,
  \ldots, s.$ Ensure that $M=[\psi_i(c_j)] $ is invertible.
  \item Let the column vector $b=[b_1, b_2, \ldots, b_s]^T$ contain the coefficients in the SPRK tableau. There
are two ways to obtain $b$:
\begin{enumerate}
\item Compute
$B_i=\int_0^1 \psi_i(\tau) d\tau$ and let $B=[B_1, B_2, \ldots,
B_s]^T.$  Then, $b=M^{-1} B.$

\item Compute a new basis set $\phi_i(\tau)$ by using the relation $\phi(\tau)=M^{-1} \psi(\tau)$,
then compute $ b=[b_1, b_2, \ldots, b_s]^T$ by $b_i=\int_0^1
\phi_i(\tau) d\tau.$
\end{enumerate}

  \item Let the matrix $A^{\phi}=[a_{ij}]$ contain the coefficients of the SPRK tableau. As before, there are two way to obtain $A^\phi$:
\begin{enumerate}
\item Compute coefficients $A^{\psi}=[A_{ij}]$, where $A_{ij}=\int_0^{c_i} \psi_j(\tau) d\tau$. Then,  the matrix is given by $A^{\phi}=[a_{ij}]= A^{\psi}M^{-T}.$
\item
   Compute $A^{\phi}=[a_{ij}]$, where $a_{ij}=\int_0^{c_i} \phi_j(\tau)
   d\tau$, directly by using the new basis functions $\phi(\cdot)=M^{-1}\psi(\cdot)$.
\end{enumerate}
\item Compute the coefficients
$\tilde{A}^{\phi}=[\tilde{a}_{ij}]$ by using
$\tilde{a}_{ij}=\frac{b_i b_j -b_j a_{ji}}{b_i}.$
\end{enumerate}
\end{boxedminipage}

\subsection{Examples}

In this subsection, we will consider four examples to illustrate the
above procedure for constructing variational integrators.

\begin{example}
We consider one-stage methods. Choose the basis function $\psi_1=1$, then for any quadrature
point $c_1$, the matrix $M=[\psi_1(c_1)]=[\,1\,]$ is invertible.
\renewcommand{\theenumi}{\roman{enumi}.}
\renewcommand{\labelenumi}{\theenumi}
\begin{enumerate}
  \item If $c_1=0,$ then $b_1=1, a_{11}=0,
\tilde{a}_{11}=1,$ which is the symplectic Euler method.
  \item  If $c_1=\frac{1}{2},$ then $b_1=1, a_{11}=\frac{1}{2},
\tilde{a}_{11}=\frac{1}{2},$ which is the midpoint rule.
  \item If $c_1=1,$ then $b_1=1, a_{11}=1,
\tilde{a}_{11}=0,$ which is the adjoint symplectic Euler method.
\end{enumerate}
\end{example}

\begin{example}
Choose basis functions $\psi_1=1,  \psi_2=\cos(\pi \tau).$ If we choose
quadrature points $c_1=0, c_2=1,$ then we obtain
\begin{align*}
M=\begin{bmatrix}
\psi_1(c_1) & \psi_1(c_2) \\
\psi_2(c_1) & \psi_2(c_2) \\
\end{bmatrix}
= 
\begin{bmatrix}
1 & 1 \\
1 & -1 \\
\end{bmatrix}
\quad\text{and}\quad
M^{-1}=\frac{1}{2}\begin{bmatrix}
1 & 1 \\
1 & -1 \\
\end{bmatrix}.
\end{align*}
One can easily compute
\begin{align*}
B_1=\int_0^1 \psi_1(\tau) d\tau =1,\qquad B_2=\int_0^1 \psi_1(\tau)
d\tau =0,
\end{align*}
Therefore we get $b=[b_1, b_2]^T=M^{-1} B=[\frac{1}{2},
\frac{1}{2}]^T$, which is the trapezoidal rule. We also compute
\begin{align*}
A^{\psi}=\begin{bmatrix}
\int_0^{c_1} \psi_1(\tau) d \tau & \int_0^{c_1} \psi_2(\tau) d \tau \\
\int_0^{c_2} \psi_1(\tau) d \tau & \int_0^{c_2} \psi_2(\tau) d \tau \\
\end{bmatrix}
= 
\begin{bmatrix}
0 & 0 \\
1 & 0 \\
\end{bmatrix}.
\end{align*}
Therefore, the matrix
\begin{align*}
A^{\phi}=\begin{bmatrix}
\int_0^{c_1} \phi_1(\tau) d \tau & \int_0^{c_1} \phi_2(\tau) d \tau \\
\int_0^{c_2} \phi_1(\tau) d \tau & \int_0^{c_2} \phi_2(\tau) d \tau \\
\end{bmatrix}
= A^{\psi} M^{-T}=
\begin{bmatrix}
0 & 0 \\
\frac{1}{2} & \frac{1}{2} \\
\end{bmatrix},
\end{align*}
 By using the relationship
$\tilde{a}_{ij}=\frac{b_i b_j -b_j a_{ji}}{b_i},$ one obtains
\begin{align*}
\tilde{A}^{\phi}=[\tilde{a}_{ij}]=
\begin{bmatrix}
0 & 0 \\
\frac{1}{2} & \frac{1}{2}
\end{bmatrix}.
\end{align*}
Thus, we obtain the St\"{o}rmer--Verlet method. Interestingly, the St\"{o}rmer--Verlet method is typically derived as a variational integrator by using linear interpolation, i.e., $\psi_1=1, \psi_2=\tau$, and the trapezoidal rule.
\end{example}
\begin{example}
If we choose basis functions $\psi_1=1, \psi_2=\cos(\pi \tau),
\psi_3=sin(\pi \tau)$ and quadrature points $c_1=0,
c_2=\frac{1}{2}, c_3=1$, we obtain a new method which is
second-order accurate, and the coefficients of the SPRK method are given by

\centerline{
\begin{tabular}{c|c c c}
$\bigstrut 0$ & $0$ & $0$ & $0$ \\
$\bigstrut \frac{1}{2}$ & $\frac{1}{4} $& $\frac{1}{\pi}$ & $\frac{\pi-4}{4 \pi}$ \\
$\bigstrut 1$& $\frac{\pi-2}{2 \pi}$ &$ \frac{2}{\pi}$ & $\frac{\pi-2}{2 \pi}$\\
\hline
& $\bigstrut \frac{\pi-2}{2 \pi}$ & $\frac{2}{\pi}$ &$\frac{\pi-2}{2 \pi} $\\
\end{tabular} \qquad\quad
\begin{tabular}{c|c c c}
$\bigstrut \frac{\pi^2-2\pi-4}{2\pi^2-4\pi}$ & $\frac{\pi-2}{2\pi}$ & $\frac{\pi-4}{\pi^2-2 \pi}$ & $0$ \\
$\bigstrut \frac{1}{2}$ & $\frac{\pi-2}{2\pi} $& $\frac{1}{\pi}$ & $0$ \\
$\bigstrut \frac{\pi^2-2\pi+4}{2\pi^2-4\pi}$& $\frac{\pi-2}{2 \pi}$ &$ \frac{1}{\pi-2}$ & $0$\\
\hline
& $\bigstrut \frac{\pi-2}{2 \pi}$ & $\frac{2}{\pi}$ &$\frac{\pi-2}{2 \pi} $\\
\end{tabular}
 }
 \end{example}
\begin{example}
Chebyshev quadrature (see p.\,415 of \cite{Hi1974}) is designed to approximate weighted integrals of the form
\begin{align*}
\int_{-1}^1 f(x) w(x) dx = b\sum_{i=1}^s f(x_i) +E[f(x)],
\end{align*}
with an equally weighted sum of the function values at the quadrature points $x_i$, and an error term $E[f(x)]$. The weight $b$ is chosen so that the quadrature is exact for $f(x)=1$, i.e., $b=\frac{1}{s} \int_{-1}^1 w(x)dx$. We are primarily interested in the case when the weight function $w(x)=1$, in which case the quadrature formula becomes
\begin{align*}
\int_{-1}^1 f(x)dx=\frac{2}{s}\sum_{i=1}^s f(x_i) +E[f(x)],
\end{align*}
where the quadrature points $x_i$  are given by the roots of polynomials (see p.\,418 of \cite{Hi1974}), the first three of which are given by 
\begin{align}\label{Gpoly}
G_0(x)=1,\quad
G_1(x)=x,\quad
G_2(x)=\frac{1}{3}(3x^2-1),\quad
G_3(x)=\frac{1}{2}(2x^3-x).
\end{align}
The error term associated with the $s$-point formula is given by
\begin{align*}
E=\begin{cases}
e_s \frac{f^{(s+1)}(\xi)}{(s+1)!} & s \text{ odd},\\
e_s \frac{f^{(s+2)}(\xi)}{(s+2)!} & s \text{ even},
\end{cases}
\qquad
\text{where}
\qquad
e_s=\begin{cases}
\int_{-1}^1 xG_s(x)dx & s \text{ odd},\\
\int_{-1}^1 x^2G_s(x)dx & s \text{ even}.
\end{cases}
\end{align*}
The error term implies that the quadrature has degree of precision $s$ for odd
$s$ and degree of precision $s+1$ for even $s$.
Note that the roots $x_i$ of the polynomials $G_i$ are in the interval $[-1, 1],$ so after a change of coordinates, we obtain quadrature points $c_i$ in the interval $[0,1]$.
Then, we use Lagrange polynomials associated with these quadrature points to construct variational integrators for $s=1, 2, 3$.
\renewcommand{\theenumi}{\roman{enumi}.}
\renewcommand{\labelenumi}{\theenumi}
\begin{enumerate}
\item one-stage, second-order method\\

\centerline{
\begin{tabular}{c|c  }
$\bigstrut \frac{1}{2}$ & $\frac{1}{2}$   \\
\hline
&$\bigstrut 1$
\end{tabular}
\qquad \quad
\begin{tabular}{c|c  }
$\bigstrut \frac{1}{2}$ & $\frac{1}{2}$   \\
\hline
&$\bigstrut 1$
\end{tabular}
}
\
\item two-stage, fourth-order method\\

\centerline{
\begin{tabular}{c|c c }
$\bigstrut \frac{1}{2}-\frac{\sqrt{3}}{6}$ & $\frac{1}{4}$ & $\frac{1}{4}-\frac{\sqrt{3}}{6}$  \\
$\bigstrut \frac{1}{2}+\frac{\sqrt{3}}{6}$& $\frac{1}{4}+\frac{  \sqrt{3}}{6}$ & $\frac{1}{4}$\\
\hline
& $\bigstrut \frac{1}{2}$ & $\frac{1}{2}$
\end{tabular}
\qquad\quad
\begin{tabular}{c|c c }
$\bigstrut \frac{1}{2}-\frac{\sqrt{3}}{6}$ & $\frac{1}{4}$ & $\frac{1}{4}-\frac{\sqrt{3}}{6}$  \\
$\bigstrut \frac{1}{2}+\frac{\sqrt{3}}{6}$& $\frac{1}{4}+\frac{  \sqrt{3}}{6}$ & $\frac{1}{4}$\\
\hline
& $\bigstrut \frac{1}{2}$ & $\frac{1}{2}$
\end{tabular}
}

\item three-stage, fourth-order method\\

\centerline{
\begin{tabular}{c|c c c}
$\bigstrut \frac{1}{2}-\frac{\sqrt{2}}{4}$ & $\frac{1}{6}+\frac{\sqrt{2}}{48}$ & $\frac{1}{6}-\frac{\sqrt{2}}{6}$ & $\frac{1}{6}-\frac{ 5\sqrt{2}}{48}$ \\
$\bigstrut \frac{1}{2}$ & $\frac{1}{6}+\frac{\sqrt{2}}{8}$ &$ \frac{1}{6}$ & $\frac{1}{6}-\frac{\sqrt{2}}{8}$\\
$\bigstrut \frac{1}{2}+\frac{\sqrt{2}}{4}$& $\frac{1}{6}+\frac{ 5 \sqrt{2}}{48}$ & $\frac{1}{6}+\frac{\sqrt{2}}{6}$ &$ \frac{1}{6}-\frac{5 \sqrt{2}}{48}$\\
\hline
&$\bigstrut \frac{1}{3}$ & $\frac{1}{3}$ &$\frac{1}{3}$
\end{tabular}
\qquad \quad
\begin{tabular}{c|c c c}
$\bigstrut \frac{1}{2}-\frac{\sqrt{2}}{4}$ & $\frac{1}{6}-\frac{\sqrt{2}}{48}$ & $\frac{1}{6}-\frac{\sqrt{2}}{8}$ & $\frac{1}{6}-\frac{ 5\sqrt{2}}{48}$ \\
$\bigstrut \frac{1}{2}$ & $\frac{1}{6}+\frac{\sqrt{2}}{6}$ &$ \frac{1}{6}$ & $\frac{1}{6}-\frac{\sqrt{2}}{6}$\\
$\bigstrut \frac{1}{2}+\frac{\sqrt{2}}{4}$& $\frac{1}{6}+\frac{ 5 \sqrt{2}}{48}$ & $\frac{1}{6}+\frac{\sqrt{2}}{8}$ &$ \frac{1}{6}+\frac{\sqrt{2}}{48}$\\
\hline
&$\bigstrut \frac{1}{3}$ & $\frac{1}{3}$ &$\frac{1}{3}$
\end{tabular}
}

\end{enumerate}
For $s=1,2$, we obtain the same methods as the ones obtained using Gauss--Legendre quadrature, which are the midpoint rule, and the two-stage, fourth-order method, respectively. For $s=3$, we obtain a three-stage SPRK that is fourth-order. The order of the SPRK methods above were determined using partitioned Runge--Kutta order theory~\cite{Bu2008}.

\end{example}
\section{Momentum Preservation and Invariance of the Discrete Right Hamiltonian
Function}\label{sec:momentum_preservation}
\paragraph{\bf Momentum Maps.} First, we recall the definition of a momentum map defined
on $T^*Q$ given in \cite{AbMa1978}.

\begin{definition}\label{def3}
Let $(P, \omega)$ be a connected symplectic manifold and $\Phi:
G\times P\rightarrow P$ be a symplectic action of the Lie group $G $ on
P, i.e., for each $g\in G$, the map $\Phi_g: P  \rightarrow P$;
$x\mapsto \Phi(g, x)$ is symplectic. We say that a map $J:
P\rightarrow \mathfrak{g}^*$, where $\mathfrak{g}^*$ is a dual space of the
Lie algebra $\mathfrak{g}$ of G, is a momentum map for the action $\Phi$
if for every $\xi \in \mathfrak{g},$ $d\hat{J}(\xi)=i_{\xi_P}
\omega$, where $\hat{J}(\xi): P \rightarrow \mathbb R$ is defined by
$\hat{J}(\xi)(x)=J(x)\cdot \xi$, and $\xi_P$ is the infinitesimal
generator of the action corresponding to $\xi.$ In other words, $J$
is a momentum map provided $X_{\hat{J}(\xi)}=\xi_P$ for all $\xi \in
\mathfrak{g}$.
\end{definition}
For our purposes, we are interested in the case $P=T^*Q$, and $\omega=dq^i \wedge dp_i$ is the canonical symplectic two-form on $T^*Q$. This gives a momentum map of the form $J: T^*Q
\rightarrow \mathfrak{g}^*$, and we describe the construction given in Theorem 4.2.10 of \cite{AbMa1978}. Notice that $\omega$ is exact, since $\omega=-d \theta=-d(p_i dq^i)$. Consider an action $\Phi_g$ that leaves the Lagrange one-form $\theta$ invariant, i.e., $\Phi_g ^* \theta =\theta$ for all $g\in G$. Then, the momentum map $J: T^*Q \rightarrow \mathfrak{g}^*$ is given by
\begin{align}\label{defmom1}
 J(x)\cdot
\xi=i_{\xi_P} \theta(x).
\end{align}
We can show that this satisfies the definition of the momentum map given above by using the fact that $\Phi_g$ leaves the one-form  $\theta$ invariant for all $g\in G$, and
$\xi_P$ is the infinitesimal generator of the action corresponding
to $\xi$. This implies that the Lie derivative of $\theta$ along the vector field $\xi_P$ vanishes, i.e., $\pounds_{\xi_P}\theta=0$, for all $\xi\in\mathfrak{g}$. By Cartan's magic formula,
$$0=\pounds_{\xi_P}\theta=i_{\xi_P} d\theta + d i_{\xi_P} \theta,$$
therefore, $d i_{\xi_P} \theta=-i_{\xi_P} d\theta=i_{\xi_P} \omega$.
As such, $\hat{J}(\xi)(x)=i_{\xi_P} \theta$ satisfies the defining property, $d\hat{J}(\xi)=i_{\xi_P}
\omega$, of a momentum map. Then, $J(x)\cdot \xi=\hat{J}(\xi)(x)=i_{\xi_P} \theta(x)$. Moreover,  by Theorem 4.2.10 of \cite{AbMa1978}, this momentum map is $\operatorname{Ad}^*$-equivariant.

Let $\Phi:G\times Q\rightarrow Q$ be an action of the Lie group $G$ on $Q$. We will give the coordinate expression for the cotangent lifted action $\Phi^{T^*Q}$. In coordinates, we denote $\Phi_{g^{-1}}: Q\rightarrow Q$ by
 $q^i=\Phi_{g^{-1}}^i(Q)$, then its cotangent lifted action $\Phi^{T^*Q}: G\times T^*Q \rightarrow T^*Q$ is given by
\begin{align}\label{defcotgroup}
\Phi^{T^*Q}(g,q,p)=T^*\Phi_{g^{-1}}(q, p)=\left(\Phi_g^i(q), p_j \frac{\partial q^j}{\partial \Phi_g^i(q)}\right),
 \end{align}
where $T^*\Phi_{g^{-1}}$ means the cotangent lift of the action
$\Phi_{g^{-1}}$.

 In the following proposition, we give the coordinate expression for the cotangent lifted action, and show that it leaves the Lagrange one-form $\theta=p_i dq^i $ invariant.

 \begin{proposition}\label{prop1}
Given an action $\Phi:G\times Q$ of a Lie group $G$ on $Q$, the cotangent lifted action $\Phi^{T^*Q}:G\times T^*Q\rightarrow T^*Q$ leaves the Lagrange one-form
$\theta=p_i dq^i$ invariant.
 \end{proposition}
\begin{proof}
Given $g\in G$, let the cotangent lifted action of $g$ on $(q,p)$ be denoted by $(Q,P)=\Phi^{T^*Q}_g(q, p)$, the components of which are given by $Q^i=\Phi_g^i(q)$
 and  $P_i=p_j \frac{\partial q^j}{\partial \Phi_g^i(q)}$. Then, a direct computation yields
\begin{align}
P_i dQ^i=P_i d\Phi_g^i(q)= p_j \frac{\partial q^j}{\partial \Phi_g^i(q)}\frac{\partial \Phi_g^i(q)}{\partial q^j} dq^j =p_j dq^j.
\end{align}
This shows that $\Phi_g^{T^*Q}$ leaves the Lagrange one-form $p_i dq^i $ invariant.
\end{proof}
Corresponding to the cotangent lift action $\Phi^{T^*Q},$ for
every $\xi \in \mathfrak{g}$, the momentum map $J: T^*Q \rightarrow
\mathfrak{g}^*$ defined in \eqref{defmom1} has the following explicit expression in
coordinates,
\begin{align}%\label{defmom2}
J(\alpha_q)\cdot
\xi=i_{\xi_P} (p_i dq^i)(\alpha_q)=p\cdot \xi_Q(q)=p\cdot
\left.\frac{d}{d\epsilon}\right|_{\epsilon=0} \Phi_{\exp(\epsilon \xi)}(q^i),
\end{align}
where $\alpha_q=(q, p) \in T^*Q$.

\paragraph{\bf Discrete Noether's Theorem.}
In the discrete case, consider the one-step discrete flow map $F_{H_d^+}: (q_0,
p_0) \mapsto (q_1, p_1)$ defined by the discrete right
Hamilton's equations,
\begin{align}\label{deffd}
p_0=D_1 H_d^+(q_0, p_1),\qquad q_1=D_2 H_d^+(q_0, p_1).
\end{align}
We will show that if the generalized discrete Lagrangian, $R_d(q_0, q_1, p_1) = p_1 q_1 - H_d^+(q_0,p_1)$, is invariant under the cotangent lifted action, then we have discrete momentum preservation, which is the discrete analogue of Noether's theorem for discrete Hamiltonian mechanics.
\begin{theorem}\label{discrete_noether}
Let $\Phi^{T^*Q}$ be  the cotangent lift action of  the action
$\Phi$ on the configuration manifold $Q$.  If the
generalized discrete Lagrangian $R_d(q_0, q_1, p_1)=p_1 q_1 -H_d^+(q_0, p_1)$ is invariant under the cotangent
lifted action $\Phi^{T^*Q},$ then the discrete flow map of the discrete right Hamilton's equations preserves the momentum map, i.e., $F_{H_d^+}^* J=J$.
\end{theorem}
\begin{proof}
In coordinates, let $(q_0^\epsilon, p_0^\epsilon):=\Phi^{T^*Q}_{\exp(\epsilon\xi)}(q_0,
p_0)$ and $(q_1^\epsilon, p_1^\epsilon):=\Phi^{T^*Q}_{\exp(\epsilon\xi)}(q_1, p_1)$.
From the invariance of $p_1 q_1 -H_d^+(q_0,
p_1),$ we have that
\begin{align}\label{dismom}
0&= \left.\frac{d}{d\epsilon}\right|_{\epsilon=0} \{p_1^\epsilon q_1^\epsilon- H_d^+(q_0^\epsilon, p_1^\epsilon)  \}\\ \nonumber
&=p_1   \left.\frac{d}{d\epsilon}\right|_{\epsilon=0} q_1^\epsilon
          +q_1   \left.\frac{d}{d\epsilon}\right|_{\epsilon=0} p_1^\epsilon
          -D_1 H_d^+(q_0, p_1)\left.\frac{d}{d\epsilon}\right|_{\epsilon=0} q_0^\epsilon
          -D_2 H_d^+(q_0, p_1)\left.\frac{d}{d\epsilon}\right|_{\epsilon=0} p_1^\epsilon \\ \nonumber
&=p_1  \left.\frac{d}{d\epsilon}\right|_{\epsilon=0} q_1^\epsilon
          +q_1   \left.\frac{d}{d\epsilon}\right|_{\epsilon=0} p_1^\epsilon
          -p_0\left.\frac{d}{d\epsilon}\right|_{\epsilon=0} q_0^\epsilon
          -q_1\left.\frac{d}{d\epsilon}\right|_{\epsilon=0} p_1^\epsilon \\ \nonumber
&=p_1  \left.\frac{d}{d\epsilon}\right|_{\epsilon=0} \Phi_{\exp(\epsilon \xi)}(q_1)
       - \left. p_0\frac{d}{d\epsilon}\right|_{\epsilon=0} \Phi_{\exp(\epsilon \xi)}(q_0)\\ \nonumber
&= p_1\cdot \xi_Q(q_1) - p_0\cdot \xi_Q(q_0),
\end{align}
where we used the discrete right Hamilton's equations \eqref{deffd} in going from the second to the third line, and $q_0^\epsilon=\Phi_{\exp(\epsilon\xi)}(q_0)$ and  $q_1^\epsilon=\Phi_{\exp(\epsilon \xi)}(q_1)$ are used in going from the third to the fourth line. Then, by the definition of $F_{H_d^+}$ and $J$, \eqref{dismom} states that $F_{H_d^+}^* J=J$.
\end{proof}

\paragraph{\bf $G$-invariant generalized discrete Lagrangians from $G$-equivariant interpolants.}
We now provide a systematic means of constructing a discrete Hamiltonian, so that the generalized discrete Lagrangian $R_d(q_0, q_1, p_1)=p_1 q_1-H_d^+(q_0,p_1)$ is $G$-invariant, provided that the generalized Lagrangian $R(q,\dot{q},p)=p\dot{q}-H(q,p)$ is $G$-invariant.

Our construction will be based on an interpolatory function $\varphi:Q^r\rightarrow C^2([0,h],Q)$, that is parameterized by $r+1$ internal points $q^\nu\in Q$, defined at the times $0=d_0 h <d_1 h<\cdots<d_r h\leq h$, i.e., $\varphi\left(d_\eta h; \{q^\nu\}_{\nu=0}^r\right)=q^\eta$. We also use a numerical quadrature formula given by quadrature weights $b_i$ and quadrature points $c_i$. We denote the momentum at the time $c_i h$ by $p^i$. Then, we construct the following discrete Hamiltonian,
\begin{equation}\label{eq:G_inv_H_d}
H_d^+(q_0, p_1)=\ext_{\substack{q^\nu\in Q, p^i\in Q^*\\q^0=q_0}} \left[p_1\cdot \varphi\left(h;\{q^\nu\}_{\nu=0}^r\right)-\sum_{i=1}^s b_i R\left(T\varphi\left(c_i h;\{q^\nu\}_{\nu=0}^r\right),p^i\right)\right],
\end{equation}
where $R(q,\dot{q},p)=p\dot{q}-H(q,p)$. An interpolatory function is $G$-equivariant if
\[\varphi(t;\{gq^\nu\}_{\nu=0}^r)=g \varphi(t;\{q^\nu\}_{\nu=0}^r).\]
Then, a $G$-invariant discrete Hamiltonian can be obtained if we use $G$-equivariant interpolatory functions.
\begin{lemma}\label{gvi:lemma:invariant_Ld}
Let $G$ be a Lie group acting on $Q$, such that gQ=Q, for all $g\in G$. If the interpolatory function $\varphi(t;\{gq^\nu\}_{\nu=0}^r)$ is
$G$-equivariant, and the generalized Lagrangian $R:TQ\oplus T^*Q\rightarrow\mathbb{R}$,
\[R(q,\dot{q},p)=p\dot{q}-H(q,p),\]
is $G$-invariant, then the generalized discrete Lagrangian $R_d:Q\times
T^*Q\rightarrow \mathbb{R}$, given by
\[R_d(q_0, q_1, p_1)=p_1 q_1 -H_d^+(q_0, p_1),\]
where,
\[H_d^+(q_0, p_1)=\ext_{\substack{q^\nu\in Q, p^i\in Q^*\\q^0=q_0}} \left[p_1\cdot \varphi\left(h;\{q^\nu\}_{\nu=0}^r\right)-\sum_{i=1}^s b_i R\left(T\varphi\left(c_i h;\{q^\nu\}_{\nu=0}^r\right),p^i\right)\right],\]
is $G$-invariant.
\end{lemma}
\begin{proof}
To streamline the notation, we denote the cotangent lifted action of $G$ on $Q$ by $\Phi_g^{T^*Q}(q,p)=(gq,gp)$. First, we note that
\begin{align*}
R_d(q_0,q_1,p_1)&=p_1q_1-\ext_{\substack{q^\nu\in Q, p^i\in Q^*\\q^0=q_0}} \left[p_1\cdot \varphi\left(h;\{q^\nu\}_{\nu=0}^r\right)-\sum_{i=1}^s b_i R\left(T\varphi\left(c_i h;\{q^\nu\}_{\nu=0}^r\right),p^i\right)\right]\\
&=\ext_{\substack{q^\nu\in Q, p^i\in Q^*\\q^0=q_0}}\left[\sum_{i=1}^s b_i R\left(T\varphi\left(c_i h;\{q^\nu\}_{\nu=0}^r\right),p^i\right)\right].
\end{align*}
Then,
\begin{align*}
R_d(g q_0, g q_1, g p_1) &= \ext_{\substack{\tilde{q}^\nu\in Q, \tilde{p}^i\in Q^*\\ \tilde{q}^0=gq_0}}\left[\sum_{i=1}^s b_i R\left(T\varphi\left(c_i h;\{\tilde{q}^\nu\}_{\nu=0}^r\right),\tilde{p}^i\right)\right]\\
&= \ext_{\substack{q^\nu\in g^{-1}Q, p^i\in g^{-1}Q^*\\ gq^0=gq_0}} \left[\sum_{i=1}^s b_i R\left(T\varphi\left(c_i h;\{gq^\nu\}_{\nu=0}^r\right),gp^i\right)\right]\\
&=\ext_{\substack{q^\nu\in Q, p^i\in Q^*\\ q^0=q_0}} \left[\sum_{i=1}^s b_i R\left(TL_g\cdot T\varphi\left(c_i h;\{q^\nu\}_{\nu=0}^r\right),gp^i\right)\right]\\
&=\ext_{\substack{q^\nu\in Q, p^i\in Q^*\\ q^0=q_0}}\left[\sum_{i=1}^s b_i R\left(T\varphi\left(c_i h;\{q^\nu\}_{\nu=0}^r\right),p^i\right)\right]\\
&=R_d(q_0, q_1, p_1),
\end{align*}
where we used the identification $\tilde{q}^\nu=g q^\nu$ in the second equality, the $G$-equivariance of the interpolatory function and the property that $gQ=Q$
in the third equality, and the $G$-invariance of the generalized Lagrangian in the fourth equality.
\end{proof}
In view of Theorem~\ref{discrete_noether}, and the above lemma, if we use a $G$-equivariant interpolatory function to construct a discrete Hamiltonian as given in \eqref{eq:G_inv_H_d}, then the discrete flow given by the discrete right Hamilton's equations will preserve the momentum map $J:T^*Q\rightarrow \mathfrak{g}^*$.\\

\paragraph{\bf Natural Charts and $G$-equivariant interpolants.}
Following the construction in \cite{MaPeSh1999}, we use the group
exponential map at the identity, $\exp_e:\mathfrak{g}\rightarrow
G$, to construct a $G$-equivariant interpolatory function, and a
higher-order discrete Lagrangian. As shown in
Lemma~\ref{gvi:lemma:invariant_Ld}, this construction yields a
$G$-invariant generalized discrete Lagrangian if the generalized Lagrangian itself is
$G$-invariant.

In a finite-dimensional Lie group $G$, $\exp_e$ is a local
diffeomorphism, and thus there is an open neighborhood $U\subset
G$ of $e$ such that $\exp_e^{-1}:U\rightarrow
\mathfrak{u}\subset\mathfrak{g}$. When the group acts on the left,
we obtain a chart $\psi_g:L_g U \rightarrow \mathfrak{u}$ at $g\in
G$ by
\[\psi_g=\exp_e^{-1}\circ L_{g^{-1}}.\]

\begin{lemma}
The interpolatory function given by
\[\varphi(g^\nu;\tau h)=\psi_{g^0}^{-1}
\Big(\sum\nolimits_{\nu=0}^s
\psi_{g^0}(g^\nu)\tilde{l}_{\nu,s}(\tau)\Big),\] is
$G$-equivariant.
\end{lemma}
\begin{proof}
\begin{align*}
\varphi(gg^\nu;\tau h) &= \psi_{(gg^0)}^{-1}
\Big(\sum\nolimits_{\nu=0}^s
\psi_{gg^0}(gg^\nu)\tilde{l}_{\nu,s}(\tau)\Big)\\
&= L_{gg^0} \exp_e \Big(\sum\nolimits_{\nu=0}^s
\exp_e^{-1}((gg^0)^{-1}(gg^\nu))\tilde{l}_{\nu,s}(\tau)\Big)\\
&= L_g L_{g^0} \exp_e \Big(\sum\nolimits_{\nu=0}^s
\exp_e^{-1}((g^0)^{-1}g^{-1}gg^\nu)\tilde{l}_{\nu,s}(\tau)\Big)\\
&= L_g \psi_{g^0}^{-1} \Big(\sum\nolimits_{\nu=0}^s
\exp_e^{-1}\circ L_{(g^0)^{-1}}(g^\nu)\tilde{l}_{\nu,s}(\tau)\Big)\\
&= L_g \psi_{g^0}^{-1} \Big(\sum\nolimits_{\nu=0}^s
\psi_{g^0}(g^\nu)\tilde{l}_{\nu,s}(\tau)\Big)\\
&= L_g \varphi(g^\nu;\tau h).\qedhere
\end{align*}
\end{proof}
This $G$-equivariant interpolatory function based on natural charts allows one to construct discrete Lie group Hamiltonian variational integrators that preserve the momentum map.

\section{Conclusions and Future Directions}
In this paper, we provided a variational characterization of the Type II generating function that generates the exact flow of Hamilton's equations, and show how this is a Type II analogue of Jacobi's solution of the Hamilton--Jacobi equation. This corresponds to the exact discrete Hamiltonian for discrete Hamiltonian mechanics, and Galerkin approximations of this lead to computable discrete Hamiltonians. In addition, we introduced a discrete Type II Hamilton--Jacobi equation, which can be viewed as a composition theorem for discrete Hamiltonians.

We introduced generalized Galerkin variational integrators from both the Hamiltonian and Lagrangian approach, and when the Hamiltonian is hyperregular, these two approach are equivalent. Furthermore, we demonstrated how these methods can be implemented as symplectic partitioned Runge--Kutta methods, and derived several examples using this framework. Finally we characterized the invariance properties of a discrete Hamiltonian which ensure that the discrete Hamiltonian flow preserves the momentum map.  

We are interested in the following topics for future work:
\begin{itemize}
\item {\it Lie--Poisson Reduction and Connections to the Hamilton--Pontryagin principle.}
Since we provided a method for constructing discrete Hamiltonians that yields a numerical method that is momentum preserving, it is natural to consider discrete analogues of Lie-Poisson reduction. In particular, the constrained variational formulation of continuous Lie-Poisson reduction~\cite{CeMaPeRa2003} appears to be related to the Hamilton--Pontryagin variational principle~\cite{YoMa2006a}. It would be interesting to develop  discrete Lie--Poisson reduction~\cite{MaPeSh1999} from the Hamiltonian perspective, in the context of the discrete Hamilton--Pontryagin principle~\cite{LeOh2008,St2009}.
\item {\it Extensions to Multisymplectic Hamiltonian PDEs.} Multisymplectic integrators have been developed in the setting of Lagrangian variational integrators~\cite{LeMaOrWe2003}, and Hamiltonian multisymplectic integrators~\cite{BrRe2001}. In the paper~\cite{MaPaSh1998}, the Lagrangian formulation of multisymplectic field theory is related to Hamiltonian multisymplectic field theory~\cite{Br1997}. It would be interesting to construct Hamiltonian variational integrators for multisymplectic PDEs by generalizing the variational characterization of discrete Hamiltonian mechanics, and the generalized Galerkin construction for computable discrete Hamiltonians, to the setting of Hamiltonian multisymplectic field theories.
\end{itemize}

\section*{Acknowledgments}
The research of ML was supported by NSF grants DMS-0726263, and CAREER Award DMS-0747659. The research of JZ was conducted in the mathematics departments at Purdue University and the University of California, San Diego, supported by a fellowship from the Academy of Mathematics and System Science, Chinese Academy of Sciences.
We would like to thank Jerrold Marsden, Tomoki Ohsawa, Joris Vankerschaver, and Matthew West for helpful discussions and comments.

\bibliography{dhvi}
\bibliographystyle{siam}
 \end{document}